\newcommand{\ad}[1]{{\bf [~AD:\ } {\em \textcolor{blue}{#1}}{\bf~]}}

\documentclass[reqno, 12pt]{amsart}
\pdfoutput=1
\makeatletter
\let\origsection=\section \def\section{\@ifstar{\origsection*}{\mysection}}
\def\mysection{\@startsection{section}{1}\z@{.7\linespacing\@plus\linespacing}{.5\linespacing}{\normalfont\scshape\centering\S}}
\makeatother

\usepackage{amsmath,amssymb,amsthm}
\usepackage{mathrsfs}
\usepackage{mathabx}\changenotsign
\usepackage{dsfont}

\usepackage{graphicx}
\usepackage{float}

\usepackage[table]{xcolor}
\newcommand*\circled[1]{\tikz[baseline=(char.base)]{\node[shape=circle,draw,inner sep=1pt] (char) {#1};}}
\usepackage{array}
\newcolumntype{C}[1]{>{\centering\arraybackslash}p{#1}}
\usepackage{hyperref}
\hypersetup{
    colorlinks,
    linkcolor={red!60!black},
    citecolor={green!60!black},
    urlcolor={blue!60!black}
}

\definecolor{codelightgray}{gray}{0.8}
\definecolor{codeverylightgray}{gray}{0.9}

\usepackage{array,multirow,colortbl,enumerate}
\usepackage{caption} 

\usepackage{graphicx}

\usepackage[open,openlevel=2,atend]{bookmark}

\usepackage[abbrev,msc-links]{amsrefs}
\usepackage{amsrefs}
\usepackage{doi}

\renewcommand{\PrintDOI}[1]{\doi{#1}}

\usepackage[T1]{fontenc}
\usepackage{lmodern}
\usepackage[babel]{microtype}
\usepackage[english]{babel}

\usepackage{geometry}
\numberwithin{equation}{section}
\numberwithin{figure}{section}

\usepackage{enumitem}

\let\polishlcross=\l
\def\l{\ifmmode\ell\else\polishlcross\fi}

\def\paragraph#1{%
  \noindent\textbf{#1.}\enspace}

\let\emptyset=\varnothing
\let\setminus=\smallsetminus

\makeatletter
\def\moverlay{\mathpalette\mov@rlay}
\def\mov@rlay#1#2{\leavevmode\vtop{   \baselineskip\z@skip \lineskiplimit-\maxdimen
   \ialign{\hfil$\m@th#1##$\hfil\cr#2\crcr}}}
\newcommand{\charfusion}[3][\mathord]{
    #1{\ifx#1\mathop\vphantom{#2}\fi
        \mathpalette\mov@rlay{#2\cr#3}
      }
    \ifx#1\mathop\expandafter\displaylimits\fi}
\makeatother

\DeclareFontFamily{U}  {MnSymbolC}{}
\DeclareSymbolFont{MnSyC}         {U}  {MnSymbolC}{m}{n}
\DeclareFontShape{U}{MnSymbolC}{m}{n}{
    <-6>  MnSymbolC5
   <6-7>  MnSymbolC6
   <7-8>  MnSymbolC7
   <8-9>  MnSymbolC8
   <9-10> MnSymbolC9
  <10-12> MnSymbolC10
  <12->   MnSymbolC12}{}
\DeclareMathSymbol{\powerset}{\mathord}{MnSyC}{180}

\usepackage{tikz}
\usetikzlibrary{calc,decorations.pathmorphing,decorations.pathreplacing}
\pgfdeclarelayer{background}
\pgfdeclarelayer{foreground}
\pgfdeclarelayer{front}
\pgfsetlayers{background,main,foreground,front}

\let\epsilon=\varepsilon
\let\eps=\epsilon
\let\rho=\varrho
\let\theta=\vartheta
\let\kappa=\varkappa

\let\E=\EE
\def\NN{{\mathds N}}

\def\PP{{\mathds P}}
\let\Prob=\PP

\newcommand{\cC}{\mathcal{C}}

\newcommand{\cF}{\mathcal{F}}

\theoremstyle{plain}
\newtheorem{thm}{Theorem}[section]
\newtheorem{theorem}[thm]{Theorem}

\newtheorem{prop}[thm]{Proposition}
\newtheorem{claim}[thm]{Claim}

\newtheorem{cor}[thm]{Corollary}
\newtheorem{lemma}[thm]{Lemma}

\theoremstyle{definition}

\newtheorem{dfn}[thm]{Definition}

\newtheorem{prob}[thm]{Problem}

\usepackage{accents}
\newcommand{\seq}[1]{\accentset{\rightharpoonup}{#1}}

\let\phi=\varphi

\begin{document}

\title[Powers of Hamiltonian cycles in randomly augmented graphs]{Powers of Hamiltonian cycles in randomly augmented Dirac graphs -- the complete collection}

\author[S. Antoniuk]{Sylwia Antoniuk}
\address{Department of Discrete Mathematics, Adam Mickiewicz University, Pozna\'n, Poland}
\email{antoniuk@amu.edu.pl}

\author[A. Dudek]{Andrzej Dudek}
\address{Department of Mathematics, Western Michigan University, Kalamazoo, MI, USA}
\email{andrzej.dudek@wmich.edu}
\thanks{The second author was supported in part by Simons Foundation Grant \#522400.}

\author[A. Ruci\'nski]{Andrzej Ruci\'nski}
\address{Department of Discrete Mathematics, Adam Mickiewicz University, Pozna\'n, Poland}
\email{\tt rucinski@amu.edu.pl}
\thanks{The third author was supported by Narodowe Centrum Nauki, grant 2018/29/B/ST1/00426}

\begin{abstract}
We  study the powers of Hamiltonian cycles in randomly augmented Dirac graphs, that is, $n$-vertex graphs $G$ with minimum degree at least $(1/2+\eps)n$ to which some random edges are added.  For any Dirac graph and \emph{every} integer $m\ge2$, we accurately estimate  the threshold probability $p=p(n)$ for the event that the random augmentation $G\cup G(n,p)$  contains the $m$-th power of a Hamiltonian cycle. 
\end{abstract}

\maketitle

\section{Introduction}

In this paper we continue the study of powers of Hamiltonian cycles in randomly augmented Dirac graphs initiated by Bohman, Frieze, and Martin in \cite{BFM2003} and developed in \cite{DRRS,ADRRS,NT} and \cite{BPSS2022}.
 For $m\in\NN$ the \emph{$m$-th power~$F^m$} of a~graph~$F$ is defined
as the graph on the same vertex set whose edges join distinct vertices at distance at most~$m$ in~$F$.
The $m$-th power of a path or a cycle will be often called  an \emph{$m$-path} or, respectively, an \emph{$m$-cycle}. Note that the $m$-cycle  on $v$ vertices has  $vm$ edges, while the $v$-vertex $m$-path has $vm-\binom{m+1}2$ edges.

A Hamiltonian cycle in a graph $G$ is a cycle which passes through all vertices of~$G$. Extending the celebrated result of Dirac~\cite{Dirac}, Koml\'os, S\'ark\"ozy, and Szemer\'edi in~\cite{KSS1996,KSS1998} proved that for large $n$ every graph $G$ with $n$ vertices and minimum degree $\delta(G)\ge\tfrac k{k+1}n$ contains the $k$-th power of a Hamiltonian cycle.

\subsection{Thresholds and over-thresholds}

For an integer $k\ge0$, a sequence of $n$-vertex graphs $G=G(n)$ is \emph{$k$-Dirac} if $\delta(G)\ge\left(\tfrac k{k+1}+\epsilon\right)n$ for a constant $\epsilon>0$ and all $n\ge n_0$, for some $n_0\in\mathbb{N}$. For integers $m>k\ge1$, a \emph{$(k,m)$-Dirac threshold} was defined in \cite[Def. 1.1]{ADRRS} with $\alpha=k/(k+1)+\epsilon$ as,  roughly, the smallest function $d(n)$ such that with probability approaching 1, for every $k$-Dirac sequence of graphs $G(n)$ and every $p=p(n)\ge C d(n)$, where $C=C(\epsilon,k)>0$, the union $G(n)\cup G(n,p)$ contains the $m$-th power of a Hamiltonian cycle. If existed, such a threshold function was denoted by $d_{k,m}(n)$. The main feature of that definition was that it was independent of the constant $\eps$ (or $\alpha$) -- only the multiplicative constants depended on $\eps$.

For $k=0$, the threshold  $d_0=n^{-1}$ has been established in \cite{BFM2003}.
In \cite{DRRS} it was proved that $d_{k,k+1}(n)=n^{-1}$ for all $k\ge1$. This was substantially extended in \cite{ADRRS} to cover many other pairs $(k,m)$. In particular, it turned out  that  $d_{k,m}(n)=n^{-1}$ for all $m\le 2k+1$. (This result was independently obtained by Nenadov and Truji\'{c} in~\cite{NT}.) For $m=2$, B\"ottcher, Parczyk, Sgueglia, and Skokan in \cite{BPSS2022} extended this result to a more general setting when $\delta(G)\ge\alpha n$, for all  values of $\alpha\in(0,1)$.

However, a vast majority of pairs $(k,m)$ were still unaccounted for. In this paper we consider the case $k=1$ only and  solve the problem \emph{for all} values of $m$. Along the way it turned out, however, that we  had to come up with a new notion of threshold.

Looking back at the results in \cite{ADRRS}  for just $k=1$, it is  perhaps surprising that only the values of $m\in\{2,3,4,5,8\}$ are covered. Indeed, Theorem 1.2 in \cite{ADRRS} established  the bound $d_{1,m}(n)\le n^{-2/\ell}$ where $\ell=\min\{\ell'\le m: \ell'\ge(m-\ell')(m-\ell'+1)\}$. On the other hand, Theorem 1.3 therein yielded the lower bound $d_{1,m}(n)\ge n^{-2/\left(\lfloor m/2\rfloor+1\right)}$, and these  bounds coincide only for $m=2,3,4$. For $m=8$ they leave a gap $n^{-2/5}\le d_{1,8}(n)\le n^{-1/3}$ which was closed by showing that $d_{1,8}(n)=n^{-1/3}$ (see \cite[Theorem 1.5]{ADRRS}).

The case $m=5$, treated only as a concluding remark in \cite{ADRRS}, was exceptional in that it escaped the notion of threshold as defined there.
Indeed, in Claims 9.1 and 9.2 in \cite{ADRRS} the exponent of $n$  depends on $\epsilon$. More precisely, for each $\eps>0$ there exist two constants $c'=c'(\eps)$ and $c''=c''(\eps)$ such that the transition  from limiting probability zero to one takes place somewhere between $n^{-1/2-c'}$ and $n^{-1/2-c''}$. Moreover, both $c'$ and $c''$ converge to 0 with $\eps\to0$. (For more discussion on thresholds for up to $m=10$, see Section \ref{conrem}\,II.)

Back then viewed  rather as an eccentricity, the case $m=5$ has now become an enlightening indication of the truth in full generality.
Indeed, it turns out that for all $m\not\in\{2,3,4,8\}$, the threshold behavior is similar to that for $m=5$. As our goal has been to determine a threshold function independent of $\eps$, we are in need of a new notion of a threshold.
So, one might view the four ``typical''  cases solved in \cite{DRRS,ADRRS} as the tip of an iceberg which, however, in its depths looks quite different than what the tip might have suggested.

Inspired by the case $m=5$, we now define a new type of threshold. For future references, we do it for all $k\ge1$, although this paper is exclusively devoted to the classical Dirac case of $k=1$. For integers $m\ge1$ and $n\geq m+2$,  the family (or property) $\cC_n^m$
consists of all $n$-vertex graphs~$G$ that contain the $m$-th power  of a Hamiltonian cycle.

We say that a function $d(n)$ is a \emph{$(k,m)$-Dirac over-threshold} if
\begin{itemize}
	\item[(i)] for every $\eps>0$ there exists $\mu>0$  such that for all $n$-vertex graphs $G$ with  $\delta(G)\ge(\tfrac k{k+1}+\eps)n$ and all $p=p(n)\ge d(n)n^{-\mu}$
	$$\lim_{n\to\infty}\Prob(G\cup G(n,p)\in{\mathcal C}^m_n)=1,$$
	and
	\item[(ii)]  for every real $\mu>0$, there exists $\epsilon>0$ and a sequence of $n$-vertex graphs $G_\epsilon=G_\epsilon(n)$ with $\delta(G_\eps)\ge\left(\tfrac k{k+1}+\epsilon\right)n$ such that for every $p=p(n)\le d(n)n^{-\mu}$
	$$\lim_{n\to\infty}\Prob(G_\eps\cup G(n,p)\in{\mathcal C}^m_n)=0.$$

\end{itemize}
If exists, the $(k,m)$-Dirac over-threshold is denoted by $\bar d_{k,m}(n)$. Obviously, the existence of the over-threshold  $\bar d_{k,m}(n)$ excludes the existence of the $(k,m)$-Dirac threshold $d_{k,n}(n)$.
The prefix ``over'' is meant to remind us that the abrupt change in behavior of the probability in question happens just ``below'' the function $d(n)$. The main difference between the $(k,m)$-Dirac threshold as defined in \cite{ADRRS} and the new notion above is that now the dependence on $\epsilon$ is much more substantial. As a drawback, however, for any given $\eps$ we do not obtain a threshold in the classical sense, but a pair of functions bounding it from both sides.

\subsection{Main results}\label{mr}

The thresholds obtained  in \cite{ADRRS} were related to the density of the so called braid graphs (see Definition~\ref{bra} and Fig. \ref{Fig.00}). Given a graph $G$, let $v_G=|V(G)|$ and $e_G=|E(G)|$. For a graph $G$ with $v_G>1$, by \emph{density} and, respectively, \emph{maximum density}, we mean
\[ d_G = \frac{e_G}{v_G-1} \ \ \text{ and } \ \ m_G= \max_{H\subseteq G,v_H>1}d_G.\]
The parameter $d_G$ is often called 1-density.

Roughly speaking, for $0\le r\le \ell$, a braid graph consists of an ordered collection of $\ell$-cliques such that any two consecutive cliques are joined by a~structure called an $r$-bridge having exactly ${r+1\choose 2}$ edges (see Definition~\ref{bri} and Fig. \ref{Fig.0}). It was crucial for the construction used in the proof in \cite{ADRRS} that
\begin{equation}\label{klr}
m= k\ell+r,
\end{equation}
so  that when $k+1$ braids are intertwined and all $\ell$-cliques at distance at most $k$, one from each braid, are fully connected to each other, each vertex has exactly $k\ell+r= m$ neighbors ahead of itself (and, by symmetry, also behind itself). That is, such a structure forms an $m$-path.

 The cases of $k$ and $m$ covered in \cite{ADRRS} were, with the exception of $(1,8)$ and $(2,14)$, exactly those for which one could find integers $\ell\ge2$ and $r\ge0$ with  $r(r+1)\leq\ell$ and such that $(k+1)(\ell-1)\le m\le k\ell+r$ (the R-H-S inequality is due to the monotonicity of  $m$-path containment with respect to $m$). The lower bound on $\ell$ implied that the densest subgraph of a braid was an $\ell$-clique and, consequently, the exponent in the $(k,m)$-Dirac threshold equalled the negated reciprocal of the density $d_{K_{\ell}}$ of the $\ell$-clique, that is, $-2/\ell$. The reason was that for $p\ge n^{-2/\ell}$ there are in $G(n,p)$ plenty of copies of the braid  (c.f., \cite[Proposition 5.8]{ADRRS}, an immediate  corollary of \cite[Theorem 2.2]{ADRRS}), a fact which was a crucial building block in the proof therein.

In this paper  we allow the opposite case, namely $\ell< r(r+1)$ (for $k=1$). It turns out that now the densest subgraph of a braid is the braid itself (see Proposition \ref{<d}). Moreover, another corollary of \cite[Theorem 2.2]{ADRRS}, namely Proposition \ref{5.8} below, implies that, again, there are plenty of copies of braids in $G(n,p)$ with appropriately adjusted $p$ (the exponent of $n$ is now the negated reciprocal of the density of the braid).

We aim at finding an optimal value of $\ell$ for which the (asymptotic) density of the braid determines the $(1,m)$-Dirac over-threshold.
For this sake we introduce a pivotal quantity $f_m(\ell)$.
For  integers $m\ge2$ and $\ell\in[m]$, let
$$f_m(\ell)=\frac1{\ell}\left(\binom \ell2+\binom{m-\ell+1}{2}\right)=\ell+\frac{m^2+m}{2\ell}-m-1.$$
Note that for $k=1$, \eqref{klr} becomes $m=\ell+r$. Thus, for $\ell\ge m/2$, $f_m(\ell)$ equals the average number of edges adjacent to the vertices of the first $\ell$-clique in the braid and, consequently, it equals the asymptotic density of the braid as the number of cliques tends to infinity (see Section \ref{upp} for details).

 Observe that for $\ell>0$
\[
f_m(\ell) = \left( \sqrt{\ell} - \frac{\sqrt{2m^2+2m}}{2\sqrt{\ell}} \right)^2 + \sqrt{2m^2+2m}-m-1
\ge \sqrt{2m^2+2m}-m-1.
\]
Thus, the function $f=f_m$ has a unique global minimum on the \emph{real} interval $(0,m]$ at
\[\lambda_m = \frac{\sqrt{2m^2+2m}}{2}.\]
Hence, for each integer $\ell\in[m]$ we have $f(\ell)\ge \min\left\{f(\lfloor\lambda_m\rfloor), f(\lceil\lambda_m\rceil)\right\}$. Let
\begin{equation}\label{ell_m}
	\ell_m\in\{\lfloor\lambda_m\rfloor, \lceil\lambda_m\rceil\}\quad\mbox{be such that}\quad f(\ell_m)=\min\left\{f(\lfloor\lambda_m\rfloor), f(\lceil\lambda_m\rceil)\right\}.
\end{equation}
As a convention, when $\lfloor\lambda_m\rfloor \neq\lceil\lambda_m\rceil$ but $f(\lfloor\lambda_m\rfloor)= f(\lceil\lambda_m\rceil)$, we always set $\ell_m = \lfloor\lambda_m\rfloor$. For example,  $\lambda_{20}=\sqrt{210}$, $\lfloor\lambda_{20}\rfloor = 14$, $\lceil\lambda_{20}\rceil = 15$, and $f(14)=8=f(15)$, so we set $\ell_{20} = 14$. In general,  $\lambda_m\ge m/2$ for all $m$, and $\ell_m\ge\lambda_m-1\ge m/2$ for $m\ge4$.

We are ready to state the main results of this paper. From now on we abbreviate $(1,m)$-Dirac (over-)threshold to just $m$-Dirac (over-)threshold and use $\bar d_{m}(n)$ for $\bar d_{1,m}(n)$.

\begin{theorem}\label{thm:lower}
For every integer $m\geq 2$ and a real $\mu>0$, there exists $\epsilon>0$ and a~sequence of $n$-vertex graphs $G_\epsilon=G_\epsilon(n)$ with $\delta(G_\epsilon)\ge(\tfrac12+\epsilon)n$ such that for every $p=p(n)\le n^{-1/f(\ell_m)-\mu}$
$$\lim_{n\to\infty}\Prob(G_\epsilon\cup G(n,p)\in{\mathcal C}^m_n)=0.$$
\end{theorem}

\noindent Theorem \ref{thm:lower} implies that, if an over-threshold exists, $\bar d_{m}(n)\ge n^{-1/f(\ell_m)}$. Below we provide an upper bound.

\begin{theorem}\label{thm:upper}
For all integers $m$ and $\ell$ such that $m/2\le\ell\le m-1$ and $\ell<(m-\ell)(m-\ell+1)$ and for all $\eps>0$ there exists $\mu>0$  such that for all $n$-vertex graphs $G$ with  $\delta(G)\ge(1/2+\eps)n$ and all $p=p(n)\ge n^{-1/f(\ell)-\mu}$
$$\lim_{n\to\infty}\Prob(G\cup G(n,p)\in{\mathcal C}^m_n)=1.$$
\end{theorem}

\noindent Theorem \ref{thm:upper} implies that, if exists, $\bar d_{m}(n)\le n^{-1/f(\ell)}$.

Remember that in \cite{DRRS,ADRRS} the $m$-Dirac thresholds for $m\in\{2,3,4,8\}$, as well as the $5$-Dirac over-threshold, have been already established. The next result states that for all other $m$, except 6 and 9, the parameter $\ell_m$ from Theorem \ref{thm:lower} satisfies the restriction on $\ell$ from Theorem \ref{thm:upper}.

\begin{prop}\label{prop:lr_ineq}
	For $m=7$ and any integer $m\geq 10$ we have
	\[
	\ell_m<(m-\ell_m)(m-\ell_m+1).
	\]
\end{prop}
\noindent We defer the simple proof of Proposition \ref{prop:lr_ineq} to Appendix~\ref{appendix:a}.

Theorems~\ref{thm:lower} and~\ref{thm:upper} together with Proposition~\ref{prop:lr_ineq} determine the $m$-Dirac over-threshold in almost all cases.

\begin{cor}\label{cor:threshold}
	For $m=7$ and every integer $m\geq 10$,
	$$\bar d_{m}(n) = n^{-1/f(\ell_m)}.$$
\end{cor}

 The  remaining cases ($m=6$ and $m=9$) are treated separately in the next theorem. It is worth mentioning that here the over-threshold differs from $ n^{-1/f(\ell_m)}$.

\begin{theorem}\label{thm:6&9}
We have
	$$\bar d_{6}(n) = n^{-4/9} \ \ \text{ and } \ \ \bar d_{9}(n) = n^{-2/7}.$$
\end{theorem}

Thereby, we have obtained the complete collection of $m$-Dirac thresholds and over-thresholds.
They are summarized in Table~\ref{table:exponents}.
\setlength{\tabcolsep}{2pt}
\begin{table}
\begin{tabular}{ | >{\columncolor{codelightgray}} c | C{1.1cm} | C{1.1cm} | C{1.1cm} | C{1.1cm} | C{1.1cm} | C{1.1cm} | C{1.1cm} | C{1.1cm} | C{1.1cm} | C{1.1cm} |}
    \hline
    \rowcolor{codelightgray}
    $m$ & 2\,\cite{DRRS} & 3\,\cite{ADRRS, NT} & 4\,\cite{ADRRS} & 5\,\cite{ADRRS} & 6 & 7 & 8\,\cite{ADRRS} & 9 & $\ge10$\\ \hline
    $\alpha_m$ & 1 & 1 & $\frac{3}{2}$ & $2$ & $\frac{9}{4}$ & $\frac{13}{5}$ & $3$ &
    $\frac{7}{2}$ & $f(\ell_m)$\\ \hline
  \end{tabular}
  \caption{The negated reciprocals $\alpha_m$ of the exponents of $n$ in the $m$-Dirac thresholds ($m\in\{2,3,4,8\}$) and over-thresholds (all other $m\ge2$).}
\label{table:exponents}
\end{table}

\section{Lower bound -- The proof of Theorem \ref{thm:lower}}

The proof of Theorem \ref{thm:lower} relies on two lemmas. The first one is a simple observation about random graphs.

\begin{lemma}\label{lem:no_K_{m+1}}
For each integer $m\geq 2$, let $\l_m$ be as defined in (\ref{ell_m}). For every $\mu>0$, if
\[p \le  n^{-1/f(\ell_m)-\mu},\]
then a.a.s.~$G(n,p)$ contains no copy of $K_{m+1}$.
\end{lemma}

\begin{proof}
 Clearly,  $\PP(G(n,p)\supset K_{m+1})=O\left(n^{m+1}p^{\binom {m+1}2}\right)=o(1)$, whenever $p=o(n^{-2/m})$. Thus, all we need to show is the inequality $f(\ell_m)\le m/2$.
 For $m\in\{2,3\}$ this can be checked by computing directly $f(\lfloor\lambda_m\rfloor)$ and $ f(\lceil\lambda_m\rceil)$. Indeed, we have $\lambda_2=\sqrt3$ and $\lambda_3=\sqrt6$, so $\ell_2=\ell_3=2$ and $f(\ell_2)=1/2<1$, while $f(\ell_3)=1<3/2$.
Assume now that $m\ge4$. It can be easily verified (by solving a quadratic inequality) that $f(\ell_m)\le m/2$  is equivalent to $\l_m\ge m/2$. Finally, recall that  $\ell_m\ge\frac{\sqrt{2m^2+2m}}{2}-1\ge m/2$ for $m\ge 4$.
\end{proof}

The proof of the second lemma, which is fully deterministic, will be given in Section~\ref{sec:3}. Recall that by an $m$-path we mean the $m$-th power of a path.

\begin{lemma}\label{lem:pathedges}
 Let $m\geq 2$ and let $P$ be an $m$-path with $V(P)=A\cup B$, $A\cap B=\emptyset$, such that there are no $m+1$ consecutive vertices in $P$ either all belonging  to $A$ or all belonging to $B$. Then, 
\[|E(P[A])| + |E(P[B])| \geq f(\ell_m)|V(P)| -2m^2,\]
where $\l_m$  is defined in (\ref{ell_m}).
\end{lemma}

\noindent
Lemma~\ref{lem:pathedges} is the key novelty compared to \cite{ADRRS} and might be of independent interest, as it provides structural information about a graph commonly arising in the study of randomly augmented graphs.

Let us now define this crucial graph. It has been first used in \cite{BFM2003}   to prove a lower bound on $d_0$, and then in many other papers in the same context, e.g., \cite{DRRS,ADRRS,NT,BPSS2022}.

\begin{dfn}\label{dfn:G_eps}
For $\eps>0$ and even $n$, let $G$ be an $n$-vertex complete bipartite graph with bipartition classes $X\cup Y$, where $|X|=|Y|=n/2$. Fix two subsets $U\subset X$ and $W\subset Y$ with $|U|=|W|=\lfloor\eps n\rfloor$. The graph $G_{\eps}$ is obtained from $G$ by adding two complete bipartite graphs with bipartitions $(U,X\setminus U)$ and $(W,Y\setminus W)$.
\end{dfn}

Now we are ready to give a short proof of Theorem \ref{thm:lower}.

\begin{proof}[Proof of Theorem~\ref{thm:lower}]
Given $m$ and $\mu$, let $\eps>0$ be any constant such that $\eps\le1/6$ and
\begin{equation}\label{epsi}
\frac1\eps >\frac{6m^2}{(f(\ell_m))^2\mu}+\frac{6m^2}{f(\ell_m)}.
\end{equation}
(Note that $\eps$ decreases when $\mu$ does.)

Let $G_{\eps}$ be as in Definition~\ref{dfn:G_eps}. Let $p\le n^{-1/f(\ell_m)-\mu}$ and suppose that $G_{\eps}\cup G(n,p)$ contains the $m$-th power of a Hamiltonian cycle $C$. After removing the vertices of $U\cup W$ from $G_\eps$ (and thus from $C$), we obtain a collection of at most $2\lfloor\eps n\rfloor$ $m$-paths in $(X\setminus U)\cup(Y\setminus W)$, hence at least one of them must be of length at least
\[\frac{n-2\lfloor\eps n\rfloor}{2\lfloor\eps n\rfloor} \ge \frac1{2\eps} - 1\ge \frac1{3\eps}.\]
Fix such a path $P$ of length $L=\lceil\tfrac1{3\eps}\rceil$ and consider its edgewise intersection with $G(n,p)$. Setting $A=V(P)\cap X$ and $B=V(P)\cap Y$, by the construction of $G_{\eps}$ and the removal of $U\cup W$,
\[ G(n,p)\cap P= P[A]\cup P[B]. \]


Let $M=\lceil f(\ell_m)L\rceil-2m^2$. By Lemma~\ref{lem:pathedges}, either $G(n,p)\supset K_{m+1}$ or there are no $m+1$ consecutive vertices in $P$ either all belonging to $A$ or all belonging to $B$, and thus
\[|E(P[A])|+|E(P[B])| \geq M.\]
The former event,  by Lemma~\ref{lem:no_K_{m+1}}, does not hold a.a.s..
The latter, in turn, implies the existence in $G(n,p)$ of a subgraph with $L$ vertices and $M$ edges. The expected number of such subgraphs, abbreviating $f=f(\ell_m)$ and using our bound on $p$ and \eqref{epsi}, can be crudely bounded from above by
$$\binom{\binom L2}Mn^Lp^{M}\le O\left(n^{L-(1/f+\mu)M}\right)= O\left(n^{L-(1/f+\mu)(fL-2m^2)}\right)= O\left(n^{2m^2/f+2m^2\mu-\mu f/(3\eps)}\right)=o(1).$$

Hence, a.a.s.\, $G_\eps\cup G(n,p)$ cannot contain such a path $P$, and thus a.a.s.
\[ G_{\eps}\cup G(n,p) \notin \mathcal{C}_n^m.\]
\end{proof}

\section{Proof of Lemma~\ref{lem:pathedges}\label{sec:3}}

This section is entirely devoted to the proof of our main lemma, namely Lemma~\ref{lem:pathedges}.

\begin{proof}[Proof of Lemma~\ref{lem:pathedges}]
The ordering of the vertices on $P$ defines a partition of $A\cup B$ into segments $S_1, S_2, \ldots, S_t$, which consist of maximal sets of consecutive vertices of $P$ lying entirely in $A$ or entirely in $B$. Thus, $A=S_1\cup S_3\cup\dots$ and $B=S_2\cup S_4\cup\dots$. We will first show that in order to bound the number of edges in $P[A]\cup P[B]$ we can consider a~different $m$-path $P'$ on the same set of vertices, but with a new partition $A'\cup B' = A\cup B$ and new segments $S'_1, S'_2, \ldots, S'_{t'}$ which alternate between $A'$ and $B'$, and such that
\begin{enumerate}
\item[(a)] for all $i=1,\dots, t'$, $|S'_i|\le m$, and
\item[(b)] for all $i=1,\dots, t'-1$,  $|S'_i|+|S'_{i+1}|\ge m$.
\end{enumerate}
Condition (a) implies that neither $P'[A']$, nor $P'[B']$, contains a copy of $K_{m+1}$, while condition (b) guarantees that all edges in $E(P'[A'])\cup E(P'[B'])$ connect only vertices lying in consecutive segments of $A'$ or consecutive segments of $B'$. The second condition simplifies the computation of the number of relevant edges as, given (b), the number of edges between segment $S_i'$ and all segments lying to the right will depend only on the size of $S_{i+1}'$. In the end, we will bound the size of $E(P[A])\cup E(P[B])$ by the size of $E(P'[A'])\cup E(P'[B'])$ minus a~correcting additive term, cf.\,\eqref{eq:1}. Then the latter, owing to properties (a) and (b), will be minimized by $f(\ell_m)|V(P')|+O(1)$, a quantity corresponding to the case when $P'[A']$ and $P'[B']$ have both a braid-like structure.

The path $P'$ and the partition $A'\cup B'$ are constructed from the path $P$ and the partition $A\cup B$ iteratively. In each step we can shift a number of vertices between the partition classes, thus changing the partition. Moreover, once we shift a vertex from one class to another, say we shift a vertex $v$ from $A'$ to $B'$, we delete all edges going from $v$ to vertices in $A'$, and we join $v$ with those vertices in $B'$ which lie at distance at most $m$ from $v$ in $P'$. Another valid operation is shifting the first vertex $u$ of a segment $S_{i+1}'$ at the end of segment $S_{i-1}'$. Here we don't change the partition, but we change the ordering of vertices in $P'$ and we need to adjust the edges so that $u$ is connected only with vertices in the same partition class and lying at distance at most $m$ from it in the new ordering.

{\bf Initiation:}
Initially $P'=P$, $A'=A$, $B'=B$ and $S_i'=S_i$ for $i=1,2,\ldots,t$. If $|S'_1|>m$, we leave the first $m$ vertices in $S'_1$ and shift the remaining ones to $S'_2$. Note that this operation changes the partition, as the shifted vertices move from $A'$ to $B'$, but it cannot increase the number of edges. Indeed, if $T$ is the segment we shift from $S_1'$ to $S_2'$, then for $k\leq \lceil|T|/2\rceil$ the number of edges incident with the $k$-th vertex from the left in $T$, which we have to remove, is at least as large as the number of edges incident with the $k$-th vertex from the right, which have to add.

Now, let $j$ be the largest index for which $|S'_1|+|S'_2|+\ldots+|S'_j|<m$. If $j=0$ or $j=1$ we do nothing. If $j\geq 2$ and, say, $S_j'\subset A'$, we merge all vertices from $S'_1\cup S'_2\cup\ldots\cup S'_j$ into one class and place it in $A'$. That is we redefine $S_1':=S'_1\cup S'_2\cup\ldots\cup S'_j$ and $S_i':=S'_{i+j-1}$, $i\ge 2$.
Note that this will increase the number of edges in $P'[A']\cup P'[B']$ by at most $(m-1)^2/4$. Indeed, if $x$ and $y$ are the numbers of vertices of $S'_1\cup S'_2\cup\ldots\cup S'_j$ in, respectively, $A'$ and $B'$, then, since $x+y<m$, after the change the number of edges increases by $xy \leq x(m-1-x)\le(m-1)^2/4$.
Note that after this initial step, regardless of what the value of $j$ was, we have $|S'_1|\leq m$. 

{\bf Iteration:} We then scan  $P'$ from left to right and check whether the segments $S'_i$, $i=1,\dots,t'$, satisfy conditions (a) and (b). If not, then we move some of the vertices, changing the ordering of $P'$ or changing the partition $A'\cup B'$.
Suppose that the segments $S'_1,S'_2,\ldots,S'_{i-1}$ of $P'$ satisfy (a) and (b). In view of above, we may assume that $i\geq 2$. If $S'_{i-1}$ was the last segment, then we are done and, if not, we proceed by considering four cases. In each case we make suitable modifications of $P'$ and the partition $V(P')=A'\cup B'$ and  bound the change of the number of edges in $P'[A']\cup P'[B']$. It turns out that it never increases except in case 3 where it can increase by at most $(m-1)^2/4$.

{\bf Case 0:} If $|S'_i|\leq m$ and $|S'_{i-1}|+|S'_i|\geq m$ then segments $S'_1, S'_2, \ldots, S'_i$ satisfy (a) and (b) and we can move to the next segment $S'_{i+1}$.

{\bf Case 1:} If $|S'_i|>m$, we leave the first $m$ vertices unchanged and move the remaining ones into $S'_{i+1}$, thus changing the partition (as the moved vertices change the side, see~Fig.~\ref{Fig.1}). If before this operation $S'_i$ was the last segment, then we simply create a new segment $S'_{i+1}$. Notice that after this step segments $S'_1,S'_2,\ldots,S'_i$ satisfy (a) and (b), since now $|S'_i|=m$, and we can move to the next segment $S'_{i+1}$.

Let $S$ be the sequence of the first $m$ vertices of $S'_i$ on $P'$ before the change and let $T$ be the sequence of the remaining vertices of $S'_i$. Notice that in this step we can only add or delete edges with exactly one endpoint in $T$. We delete edges between $T$ and vertices on the left, and we add edges between $T$ and some vertices on the right. Using the same argument as in the initiation we see that after this step the number of edges can either remain the same or drop.


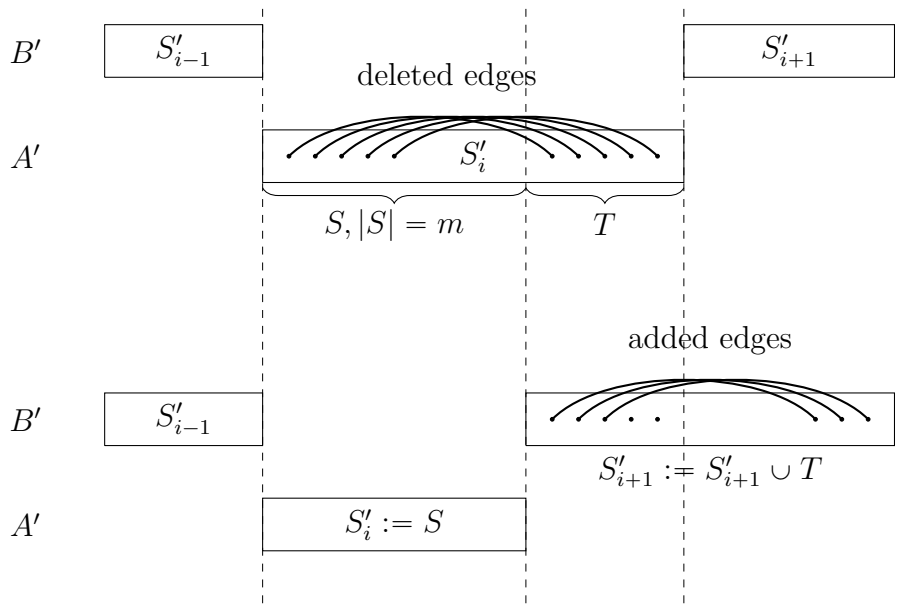
\begin{figure}

\begin{center}
\begin{tikzpicture}[scale=0.7]

\draw[dashed, ultra thin] (3,4.3) -- (3,-7);
\draw[dashed, ultra thin] (8,4.3) -- (8,-7);
\draw[dashed, ultra thin] (11,4.3) -- (11,-7);

\draw (0,3) -- (3,3) -- (3,4) -- (0,4) -- (0,3);
\node at (1.5,3.5) {$S'_{i-1}$};
\draw (3,1) -- (11,1) -- (11,2) -- (3,2) -- (3,1);
\node at (7,1.5) {$S'_i$};
\draw (11,3) -- (15,3) -- (15,4) -- (11,4) -- (11,3);
\node at (13,3.5) {$S'_{i+1}$};
\draw [decoration={brace,amplitude=0.5em},decorate]
        (8,0.9) -- (3,0.9);
\node at (5.5,0.2) {$S,|S|=m$};
\draw [decoration={brace,amplitude=0.5em},decorate]
        (11,0.9) -- (8,0.9);
\node at (9.5,0.2) {$T$};

\draw[thick] (3.5,1.5) .. controls (4.5,2.5) and (7.5,2.5) .. (8.5,1.5);
\filldraw [black] (3.5,1.5) circle (1pt);
\filldraw [black] (8.5,1.5) circle (1pt);
\draw[thick] (4,1.5) .. controls (5,2.5) and (8,2.5) .. (9,1.5);
\filldraw [black] (4,1.5) circle (1pt);
\filldraw [black] (9,1.5) circle (1pt);
\draw[thick] (4.5,1.5) .. controls (5.5,2.5) and (8.5,2.5) .. (9.5,1.5);
\filldraw [black] (4.5,1.5) circle (1pt);
\filldraw [black] (9.5,1.5) circle (1pt);
\draw[thick] (5,1.5) .. controls (6,2.5) and (9,2.5) .. (10,1.5);
\filldraw [black] (5,1.5) circle (1pt);
\filldraw [black] (10,1.5) circle (1pt);
\draw[thick] (5.5,1.5) .. controls (6.5,2.5) and (9.5,2.5) .. (10.5,1.5);
\filldraw [black] (5.5,1.5) circle (1pt);
\filldraw [black] (10.5,1.5) circle (1pt);
\node at (6.5,3) {deleted edges};

\draw (0,-4) -- (3,-4) -- (3,-3) -- (0,-3) -- (0,-4);
\node at (1.5,-3.5) {$S'_{i-1}$};
\draw (3,-6) -- (8,-6) -- (8,-5) -- (3,-5) -- (3,-6);
\node at (5.5,-5.5) {$S'_i:=S$};
\draw (8,-4) -- (15,-4) -- (15,-3) -- (8,-3) -- (8,-4);
\node at (11.5,-4.5) {$S'_{i+1}:=S'_{i+1}\cup T$};

\draw[thick] (9,-3.5) .. controls (10,-2.5) and (13,-2.5) .. (14,-3.5);
\filldraw [black] (9,-3.5) circle (1pt);
\filldraw [black] (14,-3.5) circle (1pt);
\draw[thick] (8.5,-3.5) .. controls (9.5,-2.5) and (12.5,-2.5) .. (13.5,-3.5);
\filldraw [black] (8.5,-3.5) circle (1pt);
\filldraw [black] (13.5,-3.5) circle (1pt);
\draw[thick] (9.5,-3.5) .. controls (10.5,-2.5) and (13.5,-2.5) .. (14.5,-3.5);
\filldraw [black] (9.5,-3.5) circle (1pt);
\filldraw [black] (14.5,-3.5) circle (1pt);
\filldraw [black] (10,-3.5) circle (1pt);
\filldraw [black] (10.5,-3.5) circle (1pt);
\node at (11.5,-2) {added edges};

\node at (-1.5,3.5) {$B'$};
\node at (-1.5,1.5) {$A'$};
\node at (-1.5,-3.5) {$B'$};
\node at (-1.5,-5.5) {$A'$};
\end{tikzpicture}
\end{center}

\caption{When $|S_i'|>m$, the last $|S_i'|-m$ vertices of $S_i'$ are shifted to~$S_{i+1}'$.} \label{Fig.1}
\end{figure}

{\bf Case 2:} If $|S'_{i-1}|+|S'_i|<m$ and $S'_{i+1}\neq\emptyset$, we shift the first vertex in $S'_{i+1}$ to the end of $S'_{i-1}$ (see~Fig.~\ref{Fig.2}). Notice that after this operation segments $S'_1,S'_2,\ldots, S'_{i-1}$ still satisfy (a) and (b). Moreover, if before the shift $|S'_{i+1}|=1$ while $S'_{i+2}\neq\emptyset$, then, after the shift, we merge $S'_i:=S'_i\cup S'_{i+2}$ and renumber $S'_h:=S'_{h+2}$ for $h\ge i+1$. We then check again if the assumption of Case 2 still holds for the new segments $S'_{i-1}, S'_i, S'_{i+1}$.

W.l.o.g.~we may assume that  $S'_{i+1}\subset A'$. Let $u$ be the first vertex in $S'_{i+1}$ (before the shift). Since the distance on $P'$ between any two vertices $v,w\in A'\setminus\{u\}$ does not change, the only edges of $P[A']$ affected by the shift are those incident with $u$.
Owing to the assumptions $|S'_{i-2}|+|S'_{i-1}|\geq m$ when $i>2$, and $|S'_{i-1}|+|S'_i|<m$, the number of neighbors of $u$ to the left remains unchanged by the shift (in fact, it is precisely $|S'_{i-1}|$).
 On the other hand, the distance on $P'$ between $u$ and any vertex in $A'$ lying to the right  increases (by $|S'_i|$) after the shift, so the number of edges incident with $u$ can only drop.

As for the edges in $P[B']$ affected by the shift, we need to consider only those edges $vw$ for which shifting $u$ in front of $S'_i$ has changed the distance between $v$ and $w$ on $P'$. Therefore, such edges need to be incident with $S'_i$. Notice that since $|S'_{i-1}|+|S'_i|<m$, before the shift each vertex in $S'_i$ was adjacent to some vertex in $S'_{i-2}$, and, since $|S'_{i-2}|+|S'_{i-1}|\geq m$, there were no other edges incident with $S'_i$ and going to the left. Hence, after the shift the distance between $S'_i$ and $S'_{i-2}$ increases by 1 and the number of edges between $S'_{i}$ and $S'_{i-2}$ drops by $|S'_i|$. As for the edges incident with $S'_i$ and going to the right, since after shifting $u$ the distance between vertices in $S'_i$ and vertices to the right decreases by one, we have to add at most $|S'_i|$ new edges. Therefore, again, the total number of edges in $P'[A']\cup P'[B']$ can only drop. (The merging of $S'_i$ and $S'_{i+2}$ in the case when $|S'_{i+1}|=1$ does not alter any edges  of $P$.)

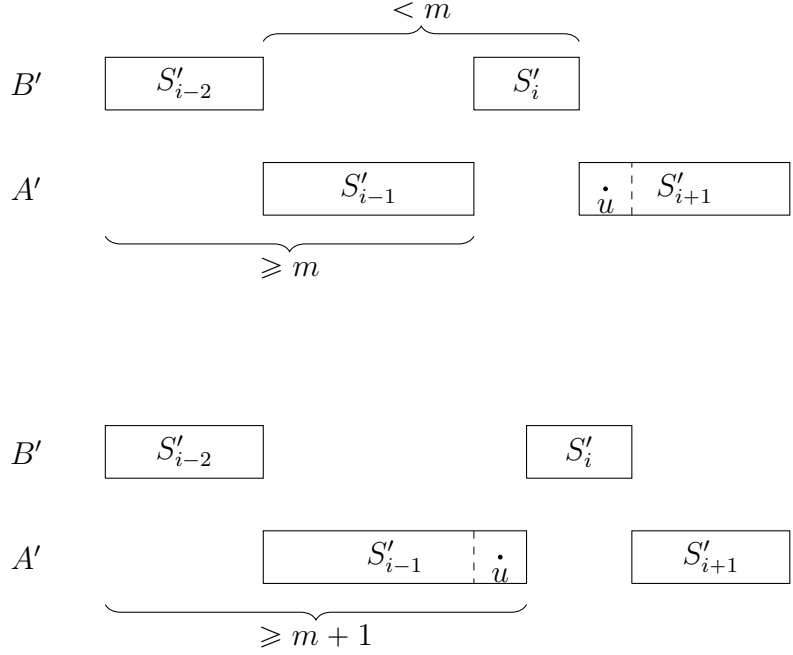
\begin{figure}
\begin{center}
\begin{tikzpicture}[scale=0.7]
\draw (0,3) -- (3,3) -- (3,4) -- (0,4) -- (0,3);
\node at (1.5,3.5) {$S'_{i-2}$};
\draw (3,1) -- (7,1) -- (7,2) -- (3,2) -- (3,1);
\node at (5,1.5) {$S'_{i-1}$};
\draw (7,3) -- (9,3) -- (9,4) -- (7,4) -- (7,3);
\node at (8,3.5) {$S'_{i}$};
\draw (9,1) -- (13,1) -- (13,2) -- (9,2) -- (9,1);
\node at (11,1.5) {$S'_{i+1}$};
\filldraw [black] (9.5,1.5) circle (1pt);
\node [below] at (9.5,1.5) {$u$};
\draw[dashed, ultra thin] (10,1) -- (10,2);

\draw [decoration={brace,amplitude=0.5em},decorate]
        (3,4.3) -- (9,4.3);
\node at (6,4.9) {$< m$};
\draw [decoration={brace,amplitude=0.5em},decorate]
        (7,0.6) -- (0,0.6);
\node at (3.5,0) {$\geq m$};

\draw (0,-4) -- (3,-4) -- (3,-3) -- (0,-3) -- (0,-4);
\node at (1.5,-3.5) {$S'_{i-2}$};
\draw (3,-6) -- (8,-6) -- (8,-5) -- (3,-5) -- (3,-6);
\node at (5.5,-5.5) {$S'_{i-1}$};
\draw (8,-4) -- (10,-4) -- (10,-3) -- (8,-3) -- (8,-4);
\node at (9,-3.5) {$S'_{i}$};
\draw (10,-6) -- (13,-6) -- (13,-5) -- (10,-5) -- (10,-6);
\node at (11.5,-5.5) {$S'_{i+1}$};
\filldraw [black] (7.5,-5.5) circle (1pt);
\node [below] at (7.5,-5.5) {$u$};
\draw[dashed, ultra thin] (7,-5) -- (7,-6);

\draw [decoration={brace,amplitude=0.5em},decorate]
        (8,-6.4) -- (0,-6.4);
\node at (4,-7) {$\geq m+1$};

\node at (-1.5,3.5) {$B'$};
\node at (-1.5,1.5) {$A'$};
\node at (-1.5,-3.5) {$B'$};
\node at (-1.5,-5.5) {$A'$};

\end{tikzpicture}
\end{center}
\caption{When $|S_{i-1}'|+|S_i'|<m$, the first vertex in $S_{i+1}'$ is shifted at the end of $S_{i-1}'$.} \label{Fig.2}
\end{figure}

{\bf Case 3 (Termination):} If $|S'_{i-1}|+|S'_i|<m$ and $S'_{i+1}=\emptyset$, we merge $S'_{i-1}:=S'_{i-1}\cup S'_i$ and finish the procedure.
By doing so we have to add all edges between $S'_{i-1}$ and $S'_i$ and, similarly as in the initiation step, the number of such edges is at most $(m-1)^2/4$.

Summing up, in the initiation step we added at most $(m-1)^2/4$ edges, then in each step we haven't increased the number of edges until maybe at the termination where we added at most $(m-1)^2/4$ edges. Thus
\begin{equation}\label{eq:1}
|E(P[A])|+|E(P[B])| \geq |E(P'[A'])|+ |E(P'[B'])| - \frac{(m-1)^2}{2}.
\end{equation}

It remains to bound the number of edges in $P'[A'] \cup P'[B']$. Suppose that the number of segments corresponding to the partition $A'\cup B'$ is $q$ and let $x_1, x_2,\ldots,x_q$ be the sizes of consecutive segments. We first count the number of edges in each segment. By (a) for each $i=1,2,\ldots,q$ we have $x_i\leq m$, hence each segment induces in $P'[A']\cup P'[B']$ exactly ${x_i\choose 2}$ edges. Next, we count the number of edges with endpoints in two different segments. By (b) for each $i=2,3,\ldots,q-1$ we have $x_{i-1}+x_i\geq m$ and $x_i+x_{i+1} \geq m$. Therefore if we consider the edges incident with $S'_{i-1}$ and going to the right, then each such edge can be incident only with $S'_{i+1}$. Moreover, the number of such edges is
\[ (m-x_i) + (m-x_i-1) + \ldots + 1 = \frac{(m-x_i)(m-x_i+1)}{2}.\]

Hence
\begin{align*}
|E(P'[A'])| + |E(P'[B'])| & = \sum_{i=1}^q {x_i\choose 2} + \sum_{i=2}^{q-1} \frac{(m-x_i)(m-x_i+1)}{2} \\
& \ge\sum_{i=1}^q \left({x_i\choose 2} +\frac{(m-x_i)(m-x_i+1)}{2}\right)  - m^2=\sum_{i=1}^qx_if(x_i)-m^2\\
& \ge\sum_{i=1}^qx_if(\ell_m)-m^2=f(\ell_m)|V(P)|-m^2.
\end{align*}
Finally, by \eqref{eq:1},
\[|E(P[A])| + |E(P[B])|\ge f(\ell_m)|V(P)|-2m^2.\]
\end{proof}

\section{Upper bound -- The proof of Theorem \ref{thm:upper}}\label{upp}

The proof of Theorem~\ref{thm:upper} is similar to that of \cite[Theorem 1.2]{ADRRS} (with $k=1$) and therefore we will only discuss the necessary amendments. In turn, the proof of \cite[Theorem 1.2]{ADRRS} followed a general outline of the proof of \cite[Theorem 1.1]{DRRS} which was based by nowadays standard method of absorption. We suggest the readers read Section 2 in \cite{DRRS} and Section 4 in \cite{ADRRS} before delving into the rest of this section.
In our proof all four pillars of the absorbing method: the Connecting Lemma, the Reservoir Lemma, the Absorbing Lemma, and the Covering Lemma, are exactly the same as, respectively, Lemmas 6.2, 4.2, 4.3, 4.4 in \cite{ADRRS}, except for the assumption on $p$.

To describe this fundamental change with respect to the proof in \cite{ADRRS}, we begin by  recalling the notions of a~bridge and a braid graph which played a crucial role therein.

\begin{dfn}\label{bri}
	For $r\ge2$, two sequences of vertices $\seq{v} = (v_1, v_2, \ldots, v_r)$ and $\seq{u} = (u_1, u_2, \ldots, u_r)$ of a graph $G$ are said to \emph{form an $r$-bridge} (or just \emph{a bridge} if the value of $r$ is clear from the context) if for each $i = 1,2, \ldots, r$, the vertex $v_i$ is adjacent in $G$ to all $u_1, u_2, \ldots, u_{i}$. 
\end{dfn}

\begin{dfn}\label{bra}
	For $t\ge1$, $\ell\ge2$, and  $1\le r\le\ell$, let $B(\ell,r,t)$ be \emph{the braid  graph} consisting of $t$ vertex-disjoint $\ell$-cliques $K_\ell^{(1)},K_\ell^{(2)},\ldots,K_\ell^{(t)}$, with vertices ordered arbitrarily,  where for each $i=1,\dots,t-1$, the last $r$ vertices of $K_\ell^{(i)}$ and the first $r$ vertices of $K_\ell^{(i+1)}$ form an  $r$-bridge. We write shortly $B_t=B(\ell,r,t)$ and for any $s\geq 1$, we denote by $sB_t$, the union of $s$ vertex disjoint copies of $B_t$.
\end{dfn}
\noindent See Fig. \ref{Fig.0} and \ref{Fig.00} for examples of a bridge and a braid.

\begin{figure}

\begin{center}
\begin{tikzpicture}[scale=1]

\draw[thick] (3.5,1.5) .. controls (4.5,3.5) and (7.5,3.5) .. (8.5,1.5);
\filldraw [black] (3.5,1.5) circle (1pt);
\filldraw [black] (8.5,1.5) circle (1pt);
\draw[thick] (4,1.5) .. controls (5,3.5) and (8,3.5) .. (9,1.5);
\draw[thick] (4,1.5) .. controls (5,3) and (7.5,3) .. (8.5,1.5);
\filldraw [black] (4,1.5) circle (1pt);
\filldraw [black] (9,1.5) circle (1pt);
\draw[thick] (4.5,1.5) .. controls (5.5,3.5) and (8.5,3.5) .. (9.5,1.5);
\draw[thick] (4.5,1.5) .. controls (5.5,3) and (8,3) .. (9,1.5);
\draw[thick] (4.5,1.5) .. controls (5.5,2.5) and (7.5,2.5) .. (8.5,1.5);
\filldraw [black] (4.5,1.5) circle (1pt);
\filldraw [black] (9.5,1.5) circle (1pt);
\draw[thick] (5,1.5) .. controls (6,3.5) and (9,3.5) .. (10,1.5);
\draw[thick] (5,1.5) .. controls (6,3) and (8.5,3) .. (9.5,1.5);
\draw[thick] (5,1.5) .. controls (6,2.5) and (8,2.5) .. (9,1.5);
\draw[thick] (5,1.5) .. controls (6,2) and (7.5,2) .. (8.5,1.5);
\filldraw [black] (5,1.5) circle (1pt);
\filldraw [black] (10,1.5) circle (1pt);

\end{tikzpicture}
\end{center}

\caption{The $4$-bridge.} \label{Fig.0}
\end{figure}
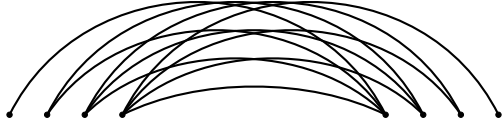

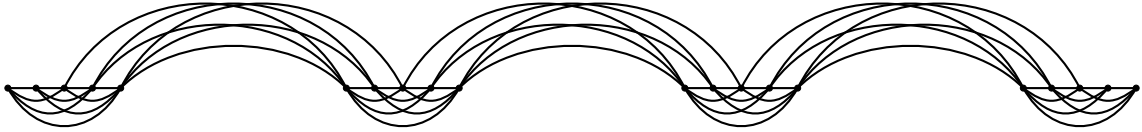
\begin{figure}
	
	\begin{center}
		\begin{tikzpicture}[scale=0.75]
		
		\filldraw [black] (2.5,1.5) circle (1.5pt);
		\filldraw [black] (3,1.5) circle (1.5pt);
		\filldraw [black] (3.5,1.5) circle (1.5pt);
		\filldraw [black] (4,1.5) circle (1.5pt);
		\filldraw [black] (4.5,1.5) circle (1.5pt);
		
		\filldraw [black] (8.5,1.5) circle (1.5pt);
		\filldraw [black] (9,1.5) circle (1.5pt);
		\filldraw [black] (9.5,1.5) circle (1.5pt);
		\filldraw [black] (10,1.5) circle (1.5pt);
		\filldraw [black] (10.5,1.5) circle (1.5pt);
		
		\filldraw [black] (14.5,1.5) circle (1.5pt);
		\filldraw [black] (15,1.5) circle (1.5pt);
		\filldraw [black] (15.5,1.5) circle (1.5pt);
		\filldraw [black] (16,1.5) circle (1.5pt);
		\filldraw [black] (16.5,1.5) circle (1.5pt);
		
		\filldraw [black] (20.5,1.5) circle (1.5pt);
		\filldraw [black] (21,1.5) circle (1.5pt);
		\filldraw [black] (21.5,1.5) circle (1.5pt);
		\filldraw [black] (22,1.5) circle (1.5pt);
		\filldraw [black] (22.5,1.5) circle (1.5pt);

		\draw[thick] (4.5,1.5) .. controls (4,0.6) and (3,0.6) .. (2.5,1.5);
		
		\draw[thick] (4,1.5) .. controls (3.5,0.9) and (3,0.9) .. (2.5,1.5);
		\draw[thick] (4.5,1.5) .. controls (4,0.9) and (3.5,0.9) .. (3,1.5);
		
		\draw[thick] (3.5,1.5) .. controls (3.1,1.2) and (2.9,1.2) .. (2.5,1.5);
		\draw[thick] (4,1.5) .. controls (3.6,1.2) and (3.4,1.2) .. (3,1.5);
		\draw[thick] (4.5,1.5) .. controls (4.1,1.2) and (3.9,1.2) .. (3.5,1.5);
		
		\draw[thick] (2.5,1.5) -- (4.5,1.5);
		
		\draw[thick] (10.5,1.5) .. controls (10,0.6) and (9,0.6) .. (8.5,1.5);
		
		\draw[thick] (10,1.5) .. controls (9.5,0.9) and (9,0.9) .. (8.5,1.5);
		\draw[thick] (10.5,1.5) .. controls (10,0.9) and (9.5,0.9) .. (9,1.5);
		
		\draw[thick] (9.5,1.5) .. controls (9.1,1.2) and (8.9,1.2) .. (8.5,1.5);
		\draw[thick] (10,1.5) .. controls (9.6,1.2) and (9.4,1.2) .. (9,1.5);
		\draw[thick] (10.5,1.5) .. controls (10.1,1.2) and (9.9,1.2) .. (9.5,1.5);
		
		\draw[thick] (8.5,1.5) -- (10.5,1.5);
		
		\draw[thick] (16.5,1.5) .. controls (16,0.6) and (15,0.6) .. (14.5,1.5);
		
		\draw[thick] (16,1.5) .. controls (15.5,0.9) and (15,0.9) .. (14.5,1.5);
		\draw[thick] (16.5,1.5) .. controls (16,0.9) and (15.5,0.9) .. (15,1.5);
		
		\draw[thick] (15.5,1.5) .. controls (15.1,1.2) and (14.9,1.2) .. (14.5,1.5);
		\draw[thick] (16,1.5) .. controls (15.6,1.2) and (15.4,1.2) .. (15,1.5);
		\draw[thick] (16.5,1.5) .. controls (16.1,1.2) and (15.9,1.2) .. (15.5,1.5);
		
		\draw[thick] (14.5,1.5) -- (16.5,1.5);
		
		\draw[thick] (22.5,1.5) .. controls (22,0.6) and (21,0.6) .. (20.5,1.5);
		
		\draw[thick] (22,1.5) .. controls (21.5,0.9) and (21,0.9) .. (20.5,1.5);
		\draw[thick] (22.5,1.5) .. controls (22,0.9) and (21.5,0.9) .. (21,1.5);
		
		\draw[thick] (21.5,1.5) .. controls (21.1,1.2) and (20.9,1.2) .. (20.5,1.5);
		\draw[thick] (22,1.5) .. controls (21.6,1.2) and (21.4,1.2) .. (21,1.5);
		\draw[thick] (22.5,1.5) .. controls (22.1,1.2) and (21.9,1.2) .. (21.5,1.5);
		
		\draw[thick] (20.5,1.5) -- (22.5,1.5);

		\draw[thick] (3.5,1.5) .. controls (4.5,3.5) and (7.5,3.5) .. (8.5,1.5);
		
		\draw[thick] (4,1.5) .. controls (5,3.5) and (8,3.5) .. (9,1.5);
		\draw[thick] (4,1.5) .. controls (5,3) and (7.5,3) .. (8.5,1.5);	
		
		\draw[thick] (4.5,1.5) .. controls (5.5,3.5) and (8.5,3.5) .. (9.5,1.5);
		\draw[thick] (4.5,1.5) .. controls (5.5,3) and (8,3) .. (9,1.5);
		\draw[thick] (4.5,1.5) .. controls (5.5,2.5) and (7.5,2.5) .. (8.5,1.5);
		
		\draw[thick] (9.5,1.5) .. controls (10.5,3.5) and (13.5,3.5) .. (14.5,1.5);
		
		\draw[thick] (10,1.5) .. controls (11,3.5) and (14,3.5) .. (15,1.5);
		\draw[thick] (10,1.5) .. controls (11,3) and (13.5,3) .. (14.5,1.5);	
		
		\draw[thick] (10.5,1.5) .. controls (11.5,3.5) and (14.5,3.5) .. (15.5,1.5);
		\draw[thick] (10.5,1.5) .. controls (11.5,3) and (14,3) .. (15,1.5);
		\draw[thick] (10.5,1.5) .. controls (11.5,2.5) and (13.5,2.5) .. (14.5,1.5);
		
		\draw[thick] (15.5,1.5) .. controls (16.5,3.5) and (19.5,3.5) .. (20.5,1.5);
		
		\draw[thick] (16,1.5) .. controls (17,3.5) and (20,3.5) .. (21,1.5);
		\draw[thick] (16,1.5) .. controls (17,3) and (19.5,3) .. (20.5,1.5);	
		
		\draw[thick] (16.5,1.5) .. controls (17.5,3.5) and (20.5,3.5) .. (21.5,1.5);
		\draw[thick] (16.5,1.5) .. controls (17.5,3) and (20,3) .. (21,1.5);
		\draw[thick] (16.5,1.5) .. controls (17.5,2.5) and (19.5,2.5) .. (20.5,1.5);

		\end{tikzpicture}
	\end{center}
	
	\caption{The $B(5,3,4)$ braid (consisting of 4 ordered cliques $K_5$ joined by 3-bridges).} \label{Fig.00}
\end{figure}

As already mentioned above, the proof in \cite{ADRRS} is based on the standard absorbing method whose four main ingredients, Connecting, Reservoir, Absorbing,  and Covering Lemmas, all claim the existence of certain $m$-paths in $G\cup G(n,p)$, either of finite length $2\ell t$ for some $t\ge2$, or easily assembled from such. Each $m$-path of length $2\ell t$ can be decomposed or embedded into the union of an $\ell$-blow-up  $P_{2t}(\ell)$ of a path of length $2t$ and the pair of braids $2B_t$, where $r=m-\ell$ (see \cite[Prop. 5.5]{ADRRS}). The proofs of the four crucial lemmas consist of a deterministic part and a probabilistic part, the former building  several copies of $P_{2t}(\ell)$ in $G$, the latter, based on \cite[Prop. 5.8]{ADRRS}, a.a.s., complementing at least one of them with a suitable copy of $2B_t$ (or its subgraph) coming from $G(n,p)$. And only the latter part of that proof has to be changed.

We now describe this change and its consequences. Recall the definitions of the density parameters $d_G$ and $m_G$ from Section \ref{mr}. Clearly, as $2B_t$ is a disjoint union of two copies of $B_t$, we have $m_{2B_t}=m_{B_t}$.
In the course of the  proof of \cite[Prop. 5.8]{ADRRS} it was shown that if $\ell\ge r(r+1)$, $r\ge1$, then for all $t\ge1$, the maximum density in a braid is obtained by an $\ell$-clique, hence $m_{B_t}=d_{K_\ell}=\ell/2$. As a result, the assumption $p\ge Cn^{-2/\ell}$ was sufficient in \cite{ADRRS} to carry on the whole proof.
Now,  under the opposite assumption $\ell<r(r+1)$, the braid graph itself achieves the maximum density among all its subgraphs, that is $m_{B_t}=d_{B_t}$.

\begin{prop}\label{<d} For all $t\ge 2$ and $1\le r\le \ell$ satisfying
\begin{equation*}
\ell<r(r+1),
\end{equation*}
we have $m_{B_t}=d_{B_t}$.
\end{prop}
\noindent We defer the proof of Proposition \ref{<d} to the end of this section.

 In view of Proposition \ref{<d}, setting $d_t=d_{B_t}$, our assumption on $p$ relaxes to $p\ge Cn^{-1/d_t}$. Most importantly,  $d_{t}$  is close to (but smaller than) $f(\ell)$.
Indeed, recall that
$$f(\ell)=\frac{\binom \ell2+\binom{r+1}2}{\ell}.$$
On the other hand,
$$d_t = \frac{t\binom \ell 2 + (t-1)\binom{r+1}2}{t\ell-1},$$
which can be rewritten as
\[
d_t = f(\ell) - \frac{(\ell-1)(r(r+1)-\ell)}{2\ell(t\ell-1)}.
\]
Hence, for $\ell<r(r+1)$, $d_t$ is a strictly increasing function of $t$  and
\begin{equation}\label{xid}
\lim_{t\to\infty}d_t=f(\ell).
\end{equation}

Next, we prove an analog of  \cite[Proposition 5.8]{ADRRS} which reflects the change in $m_{B_t}$ stemming from Proposition \ref{<d}.
For a graph $G$ with at least one edge, set
\[
\Psi_G=n^{v_G}p^{e_G}\quad\mbox{and}\quad\Phi_G=\min_{H\subseteq G, e_H>0} \Psi_H.
\]

\begin{prop}\label{5.8}
Let $t\ge1$, $r\geq 1$, $\l< r(r+1)$, $C\ge1$, and $p\ge C n^{-1/d_t}$.
Further, let $\tau>0$, $s\ge 1$,  $B=sB_t$, $F\subset B$, and $\cF$ be a family of at least $\tau n^{v_F}$ copies of $F$ in $K_n$.
Finally, let $X$ be the number of copies of $F$ belonging to $\cF$ which are present in $G(n,p)$.
Then there exists a constant $c_{F}>0$ such that
\[
\PP(X \le\tau\Psi_F/2)\le \exp\{-\tau^2c_{F}Cn\}.
\]
\end{prop}
\begin{proof}
In view of  \cite[Theorem 2.2]{ADRRS}, setting $c_F=4^{-e_F}/8$, it suffices to show that $\Phi_F\geq Cn$. First assume that $F'\subset F$ is connected, in particular, $F'\subset B$. Since $C\ge 1$, we have $np^{d_t}\ge C\ge1$. Thus,
$$\Psi_{F'}\ge \Phi_{B_t}=\min_{H\subset B_t,e_H>0}n\left(np^{d_H}\right)^{v_H-1}\ge n\left(np^{m_{B_t}}\right)= n\left(np^{d_t}\right)\ge Cn,$$
where the middle inequality follows, since $v_H\ge2$ and  $np^{d_H}\ge np^{d_t}\ge1$.
On the other hand,  if $F'=F_1\cup F_2\subset F$, where $F_1$, $F_2$ are vertex disjoint and, say, $F_1$ is connected, then, by the above bound applied to $F_1$,
$$\Psi_{F'}=\Psi_{F_1}\Psi_{F_2}\ge Cn\Psi_{F_2}> \Psi_{F_2},$$
so a disconnected $F'$ does not achieve the minimum in $\Phi_F$.
In summary,
$\Phi_F\ge Cn$ and the conclusion follows by  \cite[Theorem 2.2]{ADRRS}.
\end{proof}

We are now in position to outline the proof of Theorem~\ref{thm:upper}.
\begin{proof}[Proof of Theorem \ref{thm:upper} (outline)]
We essentially repeat the entire proof of  \cite[Theorem 1.2]{ADRRS} with $k=1$ in which all applications of \cite[Proposition 5.8]{ADRRS} are replaced with applications of Proposition \ref{5.8}.

The exponent in the threshold probability $p$ is thus determined by the largest value of $t$ with which we apply Proposition \ref{5.8}. In \cite{ADRRS}, Proposition 5.8 has been applied four times (to either $2B_t$ itself or to one of its subgraphs), namely, in the proofs of Lemma~6.2 (Connecting Lemma) with $t=4$, Lemma 4.2 (Reservoir Lemma) again with $t=4$, Proposition 7.2 (part of the proof of the Absorbing Lemma 4.3) twice with $t=2$, and, last but not least, Claim 8.2 (part of the proof of the Covering Lemma 4.4) with $t=2M$, where $M\ge 1/(4\ell \gamma^3)$ by inequality (13) in \cite{ADRRS}.

The parameter $\gamma$, which within the proof of the Covering Lemma alone, must satisfy only the restriction $\gamma\le \eps/12$,  in the main proof of \cite[Theorem 1.2]{ADRRS}, is subject to much stronger restrictions, stemming from the other three main lemmas. Therefore, it would be very cumbersome and, in fact, not necessary for us, to determine the dependence of $M$ on $\eps$  explicitly. All we can say is that $M\to\infty$ with $\eps\to 0$.

To finish the proof, it thus suffices to assume that $p\ge n^{-1/d_{2M}+\eps}$ (note that such a $p$ satisfies the assumption of Proposition \ref{5.8}),
 set $\mu=1/d_{2M}+\eps-1/f(\ell)$, so that the above assumption on $p$ becomes $p\ge n^{-1/f(\ell)-\mu}$, and observe that, by \eqref{xid}, we have  $\mu\to0$ with  $\eps\to 0$ (as then $M\to\infty$).

 Now, the proof of \cite[Theorem 1.2]{ADRRS} can be repeated mutatis mutandis.
\end{proof}

We conclude this section with the proof of Proposition \ref{<d}.
\begin{proof}[Proof of Proposition \ref{<d}] We are going to prove a slightly stronger statement, namely that $B_t$ is \emph{strictly} balanced, that is, for all \emph{proper} subgraphs $H$ of $B_t$ we have $d_H<d_t$.
First, consider the case when $\ell\in\{r,r+1\}$. Then, $B_t=B(\ell,r,t)$ is $P^r_{v}$, the $r$-th power of the $v$-vertex path, $v=t\ell$, and we have $$d_{P^r_{v}}=\frac{vr-\binom{r+1}r}{v-1}=r-\frac{\binom r2}{v-1}.$$
Note that $d_{P^r_{v}}$ is a strictly increasing function of $v$.
Therefore, it is easy to see that powers of paths are strictly balanced. Indeed, a proper subgraph $H$ of $P^r_{v}$ is either the $r$-th power of a  path $P_{v'}$, $v'<v$, or a \emph{proper spanning} subgraph of the $r$-th power of a path $P_{v'}$, $v'\le v$, so in each case $d_H<d_{P^r_{v}}$.

From now on assume that $\ell\ge r+2$. Our proof is by induction on $t$. It is easy to check that $B_1=K_\ell$ is strictly balanced.
Let $t\ge2$ and assume that $B_{t-1}$ is strictly balanced.
Recall that the vertex set $V$ of $B_t$  is split into $t$ disjoint cliques $K_\ell^{(1)},\dots,K_\ell^{(t)}$ the vertices of which are ordered (say, from left to right). Set $V_i=V(K_\ell^{(i)})$, $i=1,\dots,t$.

Let $H$ be a proper subgraph of $B_t$ such that $d_H=m_{B_t}$ and $|V(H)\cap V_t|$ is minimal. First observe that $|V(H)\cap V_t|>0$, as otherwise $H\subset B_{t-1}$, and by the induction assumption and the strict monotonicity of $d_t$, we have $m_{B_t}=d_H\le d_{t-1}<d_t$. Let $s$ be the largest index such that $V(H) \not\supset V_s$.

\begin{claim}\label{claim:s_less_t}
	We have $s \leq t-1$.
\end{claim}

This claim will be proved in Appendix~\ref{appendix:b} using tedious but elementary calculations. We now show that we can transform $H$ into a subgraph $H'$ with $d_{H'}=d_H$ and $|V(H')\cap V_t|<|V(H)\cap V_t|$, reaching a contradiction.

Clearly, $H$ is an induced subgraph of $B_t$, so it is fully determined by its vertex set. Set $X=V(H)\cap V_s$ and $Y=V(H)\cap V_t$. W.l.o.g.,  we assume that $Y$ is \emph{the leftmost} subset of $V_t$. Further, let $W\subset Y$ be \emph{the rightmost} subset of $Y$ of size $|W|=\min(|Y|, \ell-|X|) > 0$, as $|Y|= |V(H)\cap V_t| > 0$ and $\ell-|X| > 0$ by the choice of $s$. Moreover, let $W'\subset V_s\setminus X$ be \emph{the rightmost} subset of $V_s\setminus X$ with $|W'|=|W|$. We also set $Z=Y\setminus W$. We create $H'$ from $H$ by replacing $W$ with $W'$. Observe that since $|V(H')\cap V_t|<|V(H)\cap V_t|$, it remains to show that $e(H')\ge e(H)$.

\begin{figure}
\begin{center}
	\begin{tikzpicture}[scale=0.7]
		
		\draw (0,0) -- (8,0) -- (8,1) -- (0,1) -- (0,0);
		
		\node at (4,-0.6) {$X=V(H)\cap V_s$};
		
		\node at (11,2) {$V_{s+1},\ldots,V_{t-1}$};
		\node at (4,2) {$V_{s}$};
		\node at (18,2) {$V_{t}$};
		
		\node at (-0.8,0.5) {$H\!\!:$};
		\node at (-0.8,-5.5) {$H'\!\!:$};
		
		\node at (11,0.5) {$\dots$};
		\node at (11,-5.5) {$\dots$};
		
		\filldraw [black] (0.5,0.5) circle (1pt);
		\filldraw [black] (1,0.5) circle (1pt);
		\filldraw [black] (2,0.5) circle (1pt);
		\filldraw [black] (3,0.5) circle (1pt);
		\filldraw [black] (3.5,0.5) circle (1pt);
		\filldraw [black] (4,0.5) circle (1pt);
		\filldraw [black] (4.5,0.5) circle (1pt);
		\filldraw [black] (6,0.5) circle (1pt);
		\filldraw [black] (6.5,0.5) circle (1pt);
		\filldraw [black] (7,0.5) circle (1pt);
		\filldraw [black] (7.5,0.5) circle (1pt);
		
		\draw (0,-3) -- (8,-3) -- (8,-2) -- (0,-2) -- (0,-3);
		
		\filldraw [black] (1.5,-2.5) circle (1pt);
		\filldraw [black] (2.5,-2.5) circle (1pt);
		\filldraw [black] (5,-2.5) circle (1pt);
		\filldraw [black] (5.5,-2.5) circle (1pt);
		
		\node at (4,-3.6) {$W',|W'|=|W|$};
		
		\draw (0,-6) -- (8,-6) -- (8,-5) -- (0,-5) -- (0,-6);
		\draw (14,0) -- (22,0) -- (22,1) -- (14,1) -- (14,0);
		\draw (14,-6) -- (22,-6) -- (22,-5) -- (14,-5) -- (14,-6);
		\draw[dashed, ultra thin] (17.75,0) -- (17.75,1);
		\draw[dashed, ultra thin] (19.75,0) -- (19.75,1);
		\draw[dashed, ultra thin] (17.75,-6) -- (17.75,-5);
		\draw[dashed, ultra thin] (2.25,-6) -- (2.25,-5);
		
		\draw [decoration={brace,amplitude=0.5em},decorate]
		(17.75,-0.1) -- (14,-0.1);
		\node at (15.87,-0.8) {$Z$};
		
		\draw [decoration={brace,amplitude=0.5em},decorate]
		(17.75,-6.1) -- (14,-6.1);
		\node at (15.87,-6.8) {$Z$};
		
		\draw [decoration={brace,amplitude=0.5em},decorate]
		(19.75,-0.1) -- (17.75,-0.1);
		\node at (18.75,-0.8) {$W$};
		
		\draw [decoration={brace,amplitude=0.5em},decorate]
		(19.75,-1.2) -- (14,-1.2);
		\node at (16.87,-1.9) {$Y=V(H)\cap V_t$};
		
		\draw [decoration={brace,amplitude=0.5em},decorate]
		(2.25,-6.1) -- (0,-6.1);
		\node at (1.12,-6.8) {$U,|U|=|W'|$};
		
		\draw [decoration={brace,amplitude=0.5em},decorate]
		(8,-6.1) -- (2.25,-6.1);
		\node at (5.12,-6.8) {$\vec X, |\vec X|=|X|$};

		\filldraw [black] (14.5,0.5) circle (1pt);
		\filldraw [black] (15,0.5) circle (1pt);
		\filldraw [black] (15.5,0.5) circle (1pt);
		\filldraw [black] (16,0.5) circle (1pt);
		\filldraw [black] (16.5,0.5) circle (1pt);
		\filldraw [black] (17,0.5) circle (1pt);
		\filldraw [black] (17.5,0.5) circle (1pt);
		\filldraw [black] (18,0.5) circle (1pt);
		\filldraw [black] (18.5,0.5) circle (1pt);
		\filldraw [black] (19,0.5) circle (1pt);
		\filldraw [black] (19.5,0.5) circle (1pt);
		
		\filldraw [black] (14.5,-5.5) circle (1pt);
		\filldraw [black] (15,-5.5) circle (1pt);
		\filldraw [black] (15.5,-5.5) circle (1pt);
		\filldraw [black] (16,-5.5) circle (1pt);
		\filldraw [black] (16.5,-5.5) circle (1pt);
		\filldraw [black] (17,-5.5) circle (1pt);
		\filldraw [black] (17.5,-5.5) circle (1pt);
		
		\filldraw [black] (0.5,-5.5) circle (1pt);
		\filldraw [black] (1,-5.5) circle (1pt);
		\filldraw [black] (1.5,-5.5) circle (1pt);
		\filldraw [black] (2,-5.5) circle (1pt);
		\filldraw [black] (2.5,-5.5) circle (1pt);
		\filldraw [black] (3,-5.5) circle (1pt);
		\filldraw [black] (3.5,-5.5) circle (1pt);
		\filldraw [black] (4,-5.5) circle (1pt);
		\filldraw [black] (4.5,-5.5) circle (1pt);
		\filldraw [black] (5,-5.5) circle (1pt);
		\filldraw [black] (5.5,-5.5) circle (1pt);
		\filldraw [black] (6,-5.5) circle (1pt);
		\filldraw [black] (6.5,-5.5) circle (1pt);
		\filldraw [black] (7,-5.5) circle (1pt);
		\filldraw [black] (7.5,-5.5) circle (1pt);
		
	\end{tikzpicture}
\end{center}
\caption{Creating $H'$ from $H$ by replacing $W$ with $W'$. (Here $|W| = \min(|Y|, \ell-|X|) = \ell-|X|$.)}
\end{figure}
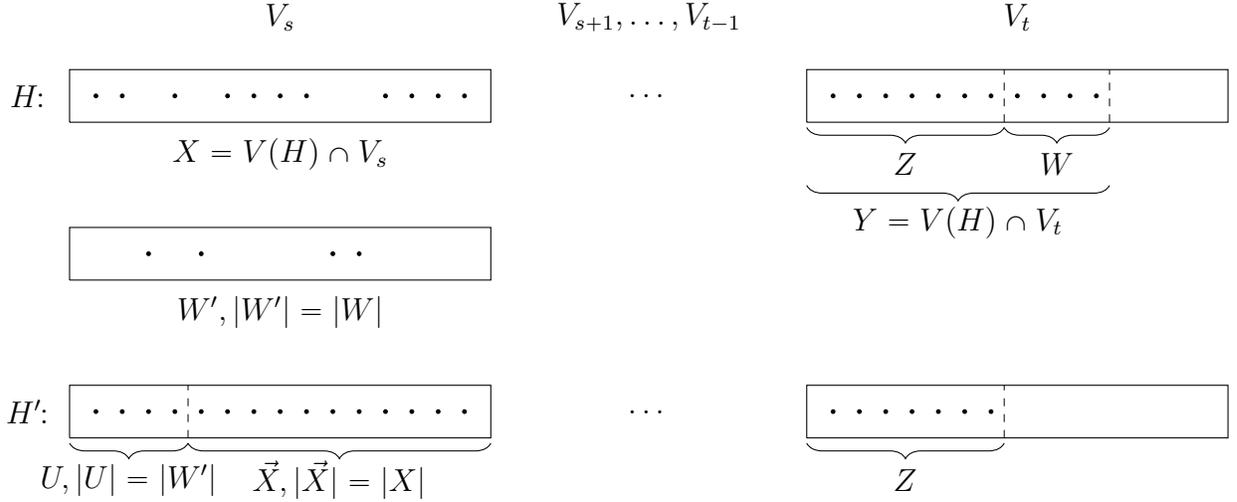

We consider two cases with respect to the value of $s$. For a vertex $v\in V$ and two disjoint subsets $U$ and $W$ of $V$ we denote by $e(U,W)$ the number of edges of $B_t$ with one endpoint in $U$ and the other in $W$, and we denote by $\deg(v,U)$ the number of edges of $B_t$ of the form $vu$, where $u\in U$.

{\bf Case $s=t-1$:} Given the definition of $W$ and $W'$, we have
$$e(H')-e(H)\ge e(W',X\cup Z)-e(W,X\cup Z).$$
We show that the latter is nonnegative.
Let $\vec X$ be \emph{the rightmost} subset of $V_s$ with $|\vec X|=|X|$ and $U$ be \emph{the rightmost} subset of $V_s\setminus\vec X$ of size $|U|=|W'|$. Then,
$$e(W', X\cup Z)\ge e(U, \vec X\cup Z),$$
because $U$ is further (or at the same distance) from the bridge between $V_{t-1}$ and $V_t$ than $W'$ was. By the same reasoning,
$$e(W,\vec X\cup Z) \ge e(W,X\cup Z).$$
It remains to show that
\begin{equation}\label{F1}
	e(U, \vec X\cup Z)\ge e(W,\vec X\cup Z).
\end{equation}
To this end, set $|U|=|W|=q$ and let $U=\{u_{q-1},\dots,u_0\}$ and $W=\{w_0,\dots,w_{q-1}\}$. Note that the indices of vertices (which follow the order imposed on $V$) grow away from the bridge. Obviously,
$$e(U, \vec X\cup Z)=\sum_{j=0}^{q-1}\deg(u_j,\vec X\cup Z)\quad\mbox{and}\quad e(W, \vec X\cup Z)=\sum_{j=0}^{q-1}\deg(w_j,\vec X\cup Z).$$
Owing to the structure of the bridge, with $x=|\vec X|=|X|$ and $z=|Z|$,
$$\deg(u_j,\vec X\cup Z)=\min(\max(x,r-j),x+z)\ \ \mbox{and}\ \ \deg(w_j,\vec X\cup Z)=\min(\max(z,r-j),x+z).$$
Since $x\ge z$, we thus have $\deg(u_j,\vec X\cup Z)\ge \deg(w_j,\vec X\cup Z)$, $j=0,\dots,q-1$, and \eqref{F1} follows.

{\bf Case  $s\le t-2$:} Now
$$e(H')-e(H)\ge e(W', X\cup V_{s+1})- e(W,Z\cup V_{t-1}),$$
and
$$e(W',X\cup V_{s+1})\ge e(U, \vec X\cup V_{s+1})\ge e(W,Z\cup V_{t-1}),$$
where the first inequality can be shown as above, and the second inequality follows from $|X|\geq |Z|$. This completes the proof.
\end{proof}

\section{Proof of Theorem \ref{thm:6&9}}


	
In order to obtain the upper bounds on $\bar d_{6}(n)$ and $\bar d_{9}(n)$,  it is enough to apply Theorem~\ref{thm:upper} with, respectively,  $m=6$, $\ell=4$, and $m=9$, $\ell=6$.

Let us then focus on the lower bounds. We are going to use a similar strategy as in the proof of Theorem~\ref{thm:lower}, where the existence of the $m$-th power of a Hamiltonian cycle implied the existence of a too dense subgraph of $G(n,p)$. Here, however, we have to bound the number of edges more carefully.

The following terminology will be useful.
Given an $m$-path $P$ and a subset $Z\subset V(P)$, we say that  $u,v\in Z$ form a \emph{$t$-far $ZZ$-pair in $P$} if there are in $P$ exactly $t-1$ vertices from $Z$ between $u$ and~$v$. If $\{u,v\}$ happens to be an edge of $P$, we then call it a \emph{$t$-far $ZZ$-edge of $P$}.


\bigskip

{\bf Case} $m=6$: Given $\mu$, let $\eps>0$ be any constant such that $\eps\le1/12$ and
\begin{equation}\label{eq:m6:eps}
\frac{9}{16\eps} > \frac{32}{9 \mu} + 8.
\end{equation}

Let $G_{\eps}$ be as in Definition~\ref{dfn:G_eps},
$p=n^{-4/9-\mu}$, and  $H=G_{\eps}\cup G(n,p)$.
 We define three  events:
 \begin{enumerate}
 \item[$\mathcal C$] -- $H$ contains the 6th power of a Hamiltonian cycle,
  \item[$\mathcal K$] -- the number of copies of $K_5$ in $G(n,p)$ is no more than $n^{5/9}$,
  \item[$\mathcal S$] -- there is in $G(n,p)$ a subgraph $F$ with $L=\lceil\tfrac1{4\eps}\rceil$ vertices and $M=\lceil\frac{9L}{4}\rceil-8$ edges.
  \end{enumerate}

  Clearly,
  \begin{equation}\label{CKS}
  \Prob(\mathcal C)\le\Prob({\mathcal C}\cap{\mathcal K})+\Prob(\neg{\mathcal K}).
  \end{equation}
  The main step of the proof will be to show that ${\mathcal C}\cap{\mathcal K}$ implies $\mathcal S$. Taking this for granted, one can quickly complete the proof. Indeed, as $p=o(n^{-4/9})$, the expected number of copies of $K_5$ in $G(n,p)$ is less than $n^5p^{10}=o(n^{5/9})$, and so, by Markov's inequality,  $\Prob(\neg{\mathcal K})=o(1)$. Similarly,  using~\eqref{eq:m6:eps}, the expected number of subgraphs of $G(n,p)$ with $L$ vertices and $M$ edges can be bounded from above by
\[
\binom{\binom L2}Mn^Lp^{M} = O\left(n^L n^{\left( -\frac{4}{9}-\mu \right) \left( \frac{9L}{4}-8\right)}\right)
= O\left(n^{\frac{32}{9}+8\mu - \frac{9}{4}L\mu}\right)
\le O\left(n^{\frac{32}{9}+8\mu - \frac{9}{16 \eps}\mu}\right) = o(1).
\]
  Thus, another application of Markov's inequality yields $\Prob({\mathcal S})=o(1)$.

  It remains to show  the inclusion ${\mathcal C}\cap{\mathcal K}\subset\mathcal S$. Let $C$ be the 6th power of a Hamilton cycle given by the event $\mathcal{C}$.
After removing from $C$ all vertices from $U\cup W$, as well as at least one vertex from each copy of at most $n^{5/9}\le \lfloor\eps n\rfloor$ copies of $K_5$ in $G(n,p)$, we obtain a~subgraph of $C$ of order at least $n-3\lfloor\eps n\rfloor$, which is a collection of at most $3\lfloor\eps n\rfloor$ 6-paths. Hence, at least one of them must have at least
\[
\frac{n-3\lfloor\eps n\rfloor}{3\lfloor\eps n\rfloor}
\ge \frac1{3\eps} - 1\ge \frac1{4\eps}
\]
vertices. (Here, the last inequality follows from the assumption $\eps \le 1/12$.)

Fix one such 6-path $P$ on $L=\lceil \frac1{4\eps}\rceil$ vertices which we relabel  as $V(P)=\{1,2,\dots,L\}$ in the order of their appearance on $P$.
Set $A=V(P)\cap X$ and $B=V(P)\cap Y$, and notice that $F=P[A]\cup P[B]\subset G(n,p)$ is $K_5$-free.
We are going to derive a few observations regarding the structure of $F$.

First, observe that both $P[A]$ and $P[B]$ contain 2-paths as spanning subgraphs. Indeed, let $u,v\in A$, $u<v$, be a $t$-far AA-pair, $t\in\{1,2\}$. As $P$ is a 6-path, it suffices to show that $v-u\le 6$. Suppose for contradiction that $v-u\ge 7$. Then, there are at least 5 vertices in $B$ between $u$ and $v$ on $P$, and  the first 5 of them,  say $w_1,\dots,w_5$, as they satisfy $w_5-w_1\le 6$, induce a copy of $K_5$ in $P[B]$, a contradiction. Hence, $P[A]$ and (by symmetry) $P[B]$ each contain a spanning 2-path.

Since the 2-path on $s$ vertices has precisely $2s-3$ edges, we conclude that the number of 1-far and 2-far edges in $F$ is
\begin{equation}\label{2L6}
\left(2|A|-3\right) + \left(2|B|-3\right) = 2L-6.
\end{equation}

Next, we are going to bound the number of 3-far edges in $F$.
Let $R$ be a segment of any 7 consecutive vertices from $P$. Since there is no $K_5$ in $F$, either $|R \cap A| = 4$ and $|R \cap B| = 3$, or $|R \cap A| = 3$ and $|R \cap B| = 4$. So, there is a 3-far $AA$-edge or 3-far $BB$-edge within $R$.
Since there are $L-6$ consecutive segments $R$ on 7 vertices and a given 3-far edge may be contained in at most four such segments, we infer that the number of 3-far edges in $F$  is at least
\begin{equation}\label{L4}
\frac{L-6}{4} \ge \frac{L}{4} - 2.
\end{equation}


Estimate \eqref{L4}, together with \eqref{2L6}, implies that  $|F|\ge\lceil\frac{9L}{4} \rceil - 8 = M$. Thus, the event $\mathcal S$ holds which finishes the proof of Theorem~\ref{thm:6&9} for $m=6$.

\bigskip


{\bf Case} $m=9$: The proof is very similar to the previous case.
Given $\mu>0$, let $p=n^{-2/7-\mu}$ and $\eps>0$ satisfy \begin{equation}\label{eq:m9:eps}
\frac{7}{8\eps} > \frac{10}{7 \mu} + 5.
\end{equation}
Let again $G_{\eps}$ be as in Definition~\ref{dfn:G_eps}. This time we define events
\begin{enumerate}
 \item[$\mathcal C$] -- $H$ contains the 9th power of a Hamiltonian cycle,
  \item[$\mathcal K$] -- the number of copies of $K_7$ in $G(n,p)$ is no more than $n/\log n$,
  \item[$\mathcal S$] -- there is in $G(n,p)$ a subgraph $F$ with $L=\lceil\tfrac1{4\eps}\rceil$ vertices and $M=\lceil\frac{7L}2\rceil-5$ edges.
  \end{enumerate}
An analog of \eqref{CKS} still holds.
The expected number of copies of~$K_7$ in $G(n,p)$ is \newline $O(n^{1-27\mu})=o(n/\log n)$, so, by Markov's inequality, $\Prob(\neg{\mathcal K})=o(1)$. Similarly,  using~\eqref{eq:m9:eps}, on can show that also $\Prob({\mathcal S})=o(1)$. It remains to show that ${\mathcal C}\cap{\mathcal K}$ implies $\mathcal S$.

Suppose that $G_{\eps}\cup G(n,p)$ contains a subgraph $C$ which is the 9th power of a~Hamiltonian cycle.
As before, after removing from $G_{\eps}\cup G(n,p)$ all vertices from $U\cup W$ as well as at least one vertex from each copy of $K_7$ in $G(n,p)$, we obtain a 9-path $P$ on $L=\lceil \frac1{4\eps}\rceil$ vertices. Let $A,B$ and $F$ be as before and let the vertices of $P$ be relabeled as $V(P)=\{1,\dots,L\}$.

Benefiting from the absence of $K_7$ in $F$, it can be shown, similarly to the case $m=6$, that both $P[A]$ and $P[B]$ contain spanning 3-paths. Indeed, let $u,v\in A$, $u<v$, be a $t$-far AA-pair, $t\in\{1,2,3\}$.  Suppose that $v-u\ge 7$. Then, there are at least 7 vertices in $B$ between $u$ and $v$ on $P$, and, consequently, a clique $K_7$ in $P[B]$, a contradiction. Hence, $P[A]$ and (by symmetry) $P[B]$ each contain a spanning 3-path.
This yields   $3L-12$ $t$-far edges in $F$ with $t\le3$. We are going to show that there are another $L/2$  edges in $F$  (4-far and 5-far).

Let $R$ be a segment of any 10 consecutive vertices from $P$. Since there is no $K_7$ in $F$, $R$ contains  5 vertices from each set $A$ and $B$, or 6 vertices from one of them (and 4 from the other). In either case there are two 4-far edges within $R$.
As there are $L-9$ segments of length 10 and each 4-far edge may belong to at most  6 of them, the total number $w$ of 4-far edges in $F$ satisfies
$$w\ge \frac{2(L-9)}6=\frac L3-3.$$

This is not quite enough, but
there may also be 5-far edges in $F$. Let us denote their number by $z$. Note that a very small value of $z$ can boost our bounds of $w+z$ even higher than a large one. In particular, it is easy to check that if $z=0$, then  $w=L-9$. So, we have to optimize by relating $w$ with $z$. Let $w_A,w_B, z_A,z_B$ be the numbers of 4-far and 5-far edges in, resp., $P[A]$ and $P[B]$.

Let $u<v$ be a 4-far $AA$-pair. The only reason for it  not to be a 4-far edge is that $v-u\ge10$, in which case there is a 5-far $BB$-edge $u'v'$ with $u<u'<v'<v$. So, we can map every  4-far $AA$-pair which is not an edge of $P[A]$ to a 5-far $BB$-edge. Note that  in this mapping the pre-image of a fixed 5-far $BB$-edge has size at most 4. Thus, $|A|-4-w_A\le4z_B$, and, by a symmetric argument,  $|B|-4-w_B\le4z_A$. Summing up, we get
$L-8-w\le 4z$, and, taking into account the previous bound on $w$, we have $w\ge\max(L/3-3,L-4z-8)$. In the end, we want to minimize the quantity
$$ w+z\ge \max(L/3-3,L-4z-8)+z.$$
It is easy to check that the minimum is achieved at $z=L/6-5/4$ and equals $L/2-17/4\ge L/2-5.$ Thus, $F$ has in total at least $3L+\lceil\frac{L}2\rceil-5$ edges and the event $\mathcal S$ holds. This completes the proof of Theorem \ref{thm:6&9}. \qed


\section{Concluding remarks}\label{conrem}

\subsection*{I} Recall that in order to determine the $m$-Dirac over-threshold $\bar d_m$, we established Theorems~\ref{thm:lower} and~\ref{thm:upper}. For a fixed $\eps>0$ and a given $n$-vertex graph $G$ with $\delta(G)\ge(1/2+\eps)n$, these theorems yield bounds on the ordinary threshold probability $\hat p$ for the property $G\cup G(n,p)\in\cC^m_n$ (assuming it exists) namely
$$n^{-1/f(\ell)-\mu_1}\le \hat p\le n^{-1/f(\ell)-\mu_2},$$
where $\mu_1=\Theta(\eps)$, while $\mu_2=\mu_2(\eps)$ is an implicit function of $\eps$ with $\lim_{\eps\to0}\mu_2(\eps)=0$ -- cf. \eqref{epsi} and, respectively, the proof of Theorem \ref{thm:upper}.
It would be nice, but probably extremely difficult, to close the gap.

\begin{prob}
For all integers $2\le\ell\le m-1$ such that $\ell<(m-\ell)(m-\ell+1)$ and for all $\eps>0$
find a function $\mu=\mu(\eps)>0$ such that for all $n$-vertex graphs $G$ with $\delta(G)\ge(1/2+\eps)n$
$$\lim_{n\to\infty}\Prob(G\cup G(n,p)\in{\mathcal C}^m_n)=\begin{cases}1\quad \mbox{if}\quad p \gg n^{-1/f(\ell)-\mu},\\0\quad \mbox{if}\quad p \ll n^{-1/f(\ell)-\mu}.\end{cases}$$
\end{prob}
\noindent (Here $a_n\ll b_n$ means $a_n=o(b_n)$, while $a_n\gg b_n$ means $b_n=o(a_n)$.)

\subsection*{II} One may wonder what is so special about the cases $m=2,3,4,8$ that just for them we have obtained the standard Dirac-threshold $d_m(n)$, while for all other values of $m$ we ended up with the Dirac over-threshold $\bar d_m(n)$. A humorous answer could be that it is yet another instance of the \emph{Law of Small Numbers}: whatever seems to be true for small instances of a parameter, fails to hold in general. Still, in this subsection we attempt to provide a more intelligent answer.

Altogether, we have had at our disposal four statements and one ad hoc technique to deduce upper and lower bounds on the order of magnitude of $p=p(n)$ which guarantees the presence of a Hamiltonian $m$-cycle in $G\cup G(n,p)$, where $G$ is a Dirac graph: Theorems \cite[Theorem 1.2]{ADRRS} and \ref{thm:upper} for upper bounds, Theorems \cite[Theorem 1.3]{ADRRS} and \ref{thm:lower} for lower bounds, the latter sometimes strengthened by forcing some extra edges in $G(n,p)$.

For $m=2,3,4$, the pair of theorems from \cite{ADRRS} suffices to pinpoint the threshold $d_m$. In particular, the lower bounds follow by observing that a Hamiltonian $m$-cycle in $G_\eps\cup G(n,p)$ must have a linear number of edges ($m=2,3$) or triangles ($m=4$) in $G(n,p)$ alone which is unlikely as $p\le c n^{-1}$ ($m=2,3$) or $p\le cn^{-2/3}$ ($m=4$). For $m=8$, \cite[Theorem 1.3]{ADRRS} does not apply, yet we have an ad hoc remedy: for $p\le cn^{-1/3}$ there are a.a.s. no copies of $K_6$ in $G(n,p)$ which implies that there have to be long 3-paths with linear number of extra 4-edges which, again, is very unlikely.

With three exceptions, for all other $m$ we have $\ell_m<(m-\ell_m)(m-\ell_m+1)$ (c.f. Proposition \ref{prop:lr_ineq}) and the pair of Theorems \ref{thm:upper} and \ref{thm:lower} ``squeezes out'' an over-threshold $\bar d_m$. Why not $d_m$? On one hand, because for such $\ell_m$ the corresponding braids are balanced (c.f. Proposition~\ref{prop:lr_ineq}) and, on the other hand, because $p$ has to be small enough to ``eliminate''  dense paths from $G(n,p)$ (c.f. Lemma \ref{lem:pathedges}).

The three remaining cases ($m=5,6,9$), at least in principle, could go either way. Let us focus on $m=5$ as the other two  are quite similar (though a bit more involved).
We have $\ell_5=4$ and $f(\ell_5)=7/4$, while $f(3)=2$ is the runner-up. Thus, Theorems \ref{thm:upper} and \ref{thm:lower} imply, respectively, that $\bar d_5(n)\le n^{-1/2}$ and $\bar d_5(n)\ge n^{-4/7}$. At this point it is still possible that for some $4/7<\tau<1/2$ we have a Dirac threshold $d_5(n)=n^{-\tau}$. However, it turned out that the lower bound can be dramatically improved to match the upper bound and thus to establish $\bar d_5(n)=n^{-1/2}$. Again, this improvement was possible by analyzing the structure of $G_\eps\cup G(n,p)$ and showing that in $G(n,p)$ alone either there are linearly many copies of $K_4$ or a long 2-path, both quite unlikely when $p\le n^{1/2-\mu}$. Finally, note that the over-thresholds for $m=5,6,9$ are \emph{not} of the form $n^{-1/f(\ell_m)}$.

For the convenience of the reader, we have summarized the above discussion in Table~\ref{table:bottom_ten}.

\setlength{\tabcolsep}{2pt}
\renewcommand{\arraystretch}{1.7}
\definecolor{Gray1}{RGB}{200,200,200}
\definecolor{Gray2}{RGB}{230,230,230}

\begin{table}[h]
{\footnotesize{
  \begin{tabular}{ | C{0.35cm} | C{0.35cm} | C{0.35cm} | c | c | c | c | c | c | c | l |}
    \hline
    \rowcolor{Gray1}
    $m$ & $\ell$ & $r$ & $f(\ell)$ & $f(\ell-1)$ & Thm 1.2\,\cite{ADRRS} & Thm 1.3\,\cite{ADRRS} & Thm 1.2 & Thm 1.1 & \parbox{2cm}{Dense\\[-5pt] structure\\[-5pt] in $G(n,p)$} & Threshold \\ \hline
    2 & 2 & 0 & n/a & n/a & \cellcolor{Gray2}{$\le n^{-1}$} & \cellcolor{Gray2}{$\ge n^{-1}$} & -- & -- & \parbox{2cm}{$\Theta(n)$ edges\\[-5pt] $\ge cn^{-1}$ } & $d_2=n^{-1}$ \\ \hline
    3 & 2 & 1 & n/a & n/a & \cellcolor{Gray2}{$\le n^{-1}$} & \cellcolor{Gray2}{$\ge n^{-1}$} & -- & -- & \parbox{2cm}{$\Theta(n)$ edges\\[-5pt] $\ge cn^{-1}$ } & $d_3=n^{-1}$ \\ \hline
    4 & 3 & 1 & n/a & n/a & \cellcolor{Gray2}{$\le n^{-2/3}$} & \cellcolor{Gray2}{$\ge n^{-2/3}$} & -- & -- & \parbox{2.2cm}{$\Theta(n)$ triangles\\[-5pt] $\ge cn^{-2/3}$ } & $d_4=n^{-2/3}$ \\ \hline
    5 & 4 & 1 & \circled{$\frac{7}{4}$} & 2 & $\le n^{-1/2}$ & -- &\cellcolor{Gray2}{$\le n^{-1/2-\mu}$} & $\ge n^{-4/7-\mu}$ & \cellcolor{Gray2}{\parbox{2cm}{2-path\\[-5pt] $\ge n^{-1/2-\mu}$ }} & $\bar{d}_5=n^{-1/2}$ \\ \hline
    6 & 5 & 1 & \circled{$\frac{11}{5}$} & $\frac{9}{4}$ & $\le n^{-2/5}$ & -- &\cellcolor{Gray2}{$\le n^{-4/9-\mu}$} & $\ge n^{-5/11-\mu}$ & \cellcolor{Gray2}{\parbox{2cm}{2-path,\\[-5pt] 3-far edges\\[-5pt] $\ge n^{-4/9-\mu}$ }} & $\bar{d}_6=n^{-4/9}$ \\ \hline
    7 & 6 & 1 & $\frac{8}{3}$ & \circled{$\frac{13}{5}$} & $\le n^{-1/3}$ & -- &\cellcolor{Gray2}{$\le n^{-5/13-\mu}$} & \cellcolor{Gray2}{$\ge n^{-5/13-\mu}$} & not needed & $\bar{d}_7=n^{-5/13}$ \\ \hline
    8 & 6 & 2 & \circled{\;\,3\;\,} & $\frac{16}{5}$ & \cellcolor{Gray2}{$\le n^{-1/3}$} & -- &$\le n^{-5/16-\mu}$ & $\ge n^{-1/3-\mu}$ & \cellcolor{Gray2}{\parbox{2cm}{3-path,\\[-5pt] 4-far edges\\[-5pt] $\ge cn^{-1/3}$ }} & $d_8=n^{-1/3}$ \\ \hline
    9 & 7 & 2 & \circled{$\frac{24}{7}$} & $\frac{7}{2}$ & $\le n^{-2/7}$ & -- & \cellcolor{Gray2}{$\le n^{-2/7-\mu}$} & $\ge n^{-7/24-\mu}$ & \cellcolor{Gray2}{\parbox{2cm}{3-path,\\[-5pt] 4-,5-far edges\\[-5pt] $\ge n^{-2/7-\mu}$}} & $\bar{d}_9=n^{-2/7}$ \\ \hline
    10 & 8 & 2 & $\frac{31}{8}$ & \circled{$\frac{27}{4}$} & $\le n^{-1/4}$ & -- &\cellcolor{Gray2}{$\le n^{-4/27-\mu}$} & \cellcolor{Gray2}{$\ge n^{-4/27-\mu}$} & not needed & $\bar{d}_{10}=n^{-4/27}$ \\ \hline
  \end{tabular}
  }}
  \caption{Summary of all thresholds for $2\le m \le 10$. Here $\ell$ is the smallest integer satisfying $\ell\ge r(r+1)$, where $r=m-\ell$; circled fractions correspond to $\ell_m$; shaded boxes indicate which theorems and ad hoc techniques (column ``Dense structure in $G(n,p)$'') determine the threshold.}
  \label{table:bottom_ten}
\end{table}

\subsection*{III}
This paper is exclusively devoted to the classical Dirac case $k=1$. However, as it was already mentioned in Introduction, in~\cite[Theorems 1.3 and 1.5]{ADRRS} the authors determined the usual $(k,m)$-Dirac thresholds $d_{k,m}=n^{-2/\ell}$ for  every $k\ge 1$ and all $m$ falling into the interval $(k+1)(\ell-1)\le m\le k\ell+(\sqrt{4\ell-1}-1)/2$ (for a fixed $k$ there is only a finite number of feasible choices of $m$), as well as $d_{1,8}=d_{2,14}=n^{-1/3}$.
We believe that, as in the case $k=1$ (c.f. discussion in Subsection II above), for all other pairs $(k,m)$, the $(k,m)$-Dirac thresholds $d_{k,m}$ do not exist and should be replaced by the over-thresholds~$\bar d_{k,m}$.

\begin{prob}\label{62}
Determine the over-threshold $\bar d_{k,m}$ for all pairs $(k,m)$, $k\ge2$, not covered in \cite{ADRRS}.
\end{prob}

We anticipate that the main difficulty in generalizing our results to  arbitrary $k\ge1$ would be, again, to formulate and prove
a $k$-analog of Lemma~\ref{lem:pathedges}.

\subsection*{Acknowledgement} We are extremely grateful to both referees for their  meticulous reading of the manuscript and numerous invaluable comments which have led to a  substantial improvement of our submission. We would also like to thank Christian Reiher for quickly rebutting  a conjecture we used to have with respect to Problem \ref{62}.







\appendix

\section{Proof of Proposition~\ref{prop:lr_ineq}}\label{appendix:a}

First we check by hand the statement for $m\in\{7,10,11,\dots,14\}$ (see Table~\ref{table:inequalities}).
	\renewcommand{\arraystretch}{1}
	\begin{table}
		\begin{tabular}{ |>{\columncolor{codelightgray}}c| C{1cm} | C{1cm} | C{1cm} | C{1cm} | C{1cm} | C{1cm} |}
			\hline
			\rowcolor{codelightgray}
			$m$ & 7 & 10 & 11 & 12 & 13 & 14\\ \hline
			$\lambda_m$ & $2\sqrt{7}$ & $\sqrt{55}$ & $\sqrt{66}$ & $\sqrt{78}$ & $\sqrt{91}$ & $\sqrt{105}$\\ \hline
			$\lfloor \lambda_m \rfloor$ & 5 & 7 & 8 & 8 & 9 & 10\\ \hline
			$\lceil \lambda_m \rceil$ & 6 & 8 & 9 & 9 & 10 & 11\\ \hline
			$f_m(\lfloor \lambda_m \rfloor)$ & $\frac{13}{5}$ & $\frac{27}{7}$ & $\frac{17}{4}$ & $\frac{19}{4}$ & $\frac{46}{9}$ & $\frac{11}{2}$\\ \hline
			$f_m(\lceil \lambda_m \rceil)$ & $\frac{8}{3}$ & $\frac{31}{8}$ & $\frac{13}{3}$ & $\frac{14}{3}$ & $\frac{51}{10}$ & $\frac{61}{11}$ \\ \hline
			\rowcolor{codeverylightgray}
			$\ell_m$ & 5 & 7 & 8 & 9 & 10 & 10\\ \hline
			$r_m=m-\ell_m$ & 2 & 3 & 3 & 3 & 3 & 4\\ \hline
			\rowcolor{codeverylightgray}
			$r_m(r_m+1)$ & 6 & 12 & 12 & 12 & 12 & 20\\ \hline
		\end{tabular}
		\caption{Verifying inequality $\ell_m<(m-\ell_m)(m-\ell_m+1)$ for $m\in\{7,10,11,\dots,14\}$.}
		\label{table:inequalities}
	\end{table}
	As for $m\ge 15$, we will show that both
	\begin{equation}\label{eq:fact12:1}
		\lfloor\lambda_m\rfloor < (m-\lfloor\lambda_m\rfloor)(m-\lfloor\lambda_m\rfloor+1)
	\end{equation}
	and
	\begin{equation}\label{eq:fact12:2}
		\lceil\lambda_m\rceil < (m-\lceil\lambda_m\rceil)(m-\lceil\lambda_m\rceil+1).
	\end{equation}

	Since $\lfloor\lambda_m\rfloor \le \lambda_m$, inequality~\eqref{eq:fact12:1} will follow from
	\[
	\lambda_m < (m-\lambda_m)(m-\lambda_m+1).
	\]
	Observe that, since $\lambda_m=\sqrt{2m^2+2m}/2$, we have
	\[
	(m-\lambda_m)(m-\lambda_m+1) - \lambda_m
	=  \left( \frac{3m}{2} -  \sqrt{2m(m+1)}\right) (m+1).
	\]
	Thus, since $\frac{3m}{2} -  \sqrt{2m(m+1)} > 0$ for $m\ge 9$, the inequality \eqref{eq:fact12:1} holds.
	
	Next, since $\lceil\lambda_m\rceil < \lambda_m+1$, inequality \eqref{eq:fact12:2} will follow from
	\[
	\lambda_m+1 \le (m-\lambda_m-1)(m-\lambda_m).
	\]
	Again, after some calculations, we get that
	\[
	(m-\lambda_m-1)(m-\lambda_m) - \lambda_m -1
	=  \left( \frac{3m}{2} -  \sqrt{2m(m+1)} -\frac{1}{2}\right)m -1.
	\]
	Due to the arithmetic-geometric mean inequality we obtain that $\sqrt{m(m+1)} \le (2m+1)/2$. Consequently,
	\begin{align*}
		\left( \frac{3m}{2} -  \sqrt{2m(m+1)} -\frac{1}{2}\right)m -1 &\ge
		\left( \frac{3m}{2} -  \sqrt{2}\cdot \frac{2m+1}{2} -\frac{1}{2}\right)m -1\\
		&= \left( \frac{3}{2}-\sqrt{2}\right)m^2 - \frac{\sqrt{2}+1}{2} m -1.
	\end{align*}
	One can easily check that the latter quadratic function of $m$ is positive for $m\ge 15$. This completes the proof of~\eqref{eq:fact12:2}. \qed

\section{Proof of Claim~\ref{claim:s_less_t}}\label{appendix:b}

We can assume that $r<\ell-1$, as the case $\ell\in\{r,r+1\}$ was already discussed at the beginning of the proof of Proposition~\ref{<d}. Recall that $t\geq 2$. Suppose that $s=t$, that is $V_i\subset H$ for $i=1,2,\dots,t-1$, and set $x=|V(H)\cap V_t|$. Notice that the maximum number of edges with at least one endpoint in $V(H)\cap V_t$ is achieved when all $x$ vertices form the leftmost subset of~$V_t$. Therefore, we assume this throughout.

{\bf Case $0 \leq x \leq r$:} Our aim is to show that
\[d_t = \frac{t\binom \ell 2 + (t-1)\binom{r+1}2}{t\ell-1} > \frac{(t-1){\binom \ell 2} + (t-2){\binom {r+1} 2} + rx}{(t-1)\ell + x-1} = d_H,\]
where the second equality uses the fact that by our assumptions $H$ consists of $t-1$ copies of $K_\ell$, such that each consecutive two of them are joined by an $r$-bridge and there is an additional $x$-tuple of vertices in $V_t$ in which each vertex has exactly $r$ neighbours to the left of it in $H$.

Consider the function
\begin{align*}
	f(\ell,r,t,x) = \left(t\binom \ell 2 + (t-1)\binom{r+1}2 \right)\left((t-1)\ell + x-1\right) \\
	- \left((t-1){\binom \ell 2} + (t-2){\binom {r+1} 2} + rx\right)\left(t\ell-1 \right).
\end{align*}

If we show that for $\ell, r, t, x$ satisfying our assumptions $f(\ell,r,t,x)>0$, we will be done. Notice that $f(\ell,r,t,x)$ is a linear function with respect to $x$, hence it is enough to verify that $f(\ell,r,t,0)>0$ and $f(\ell,r,t,r)>0$. We have
\begin{align*}
	f(\ell,r,t,0) & = \left(t\binom \ell 2 + (t-1)\binom{r+1}2 \right)\left(\ell t-\ell-1\right) \\
	& - \left((t-1){\binom\ell 2} + (t-2){\binom{r+1}2} \right)\left(\ell t-1 \right) \\
	& = - \binom\ell 2 + \binom{r+1}2 \left(\ell-1 \right) = \frac{\ell-1}{2}\left(r(r+1)-\ell \right) > 0,
\end{align*}
since $r(r+1)>\ell$. Similarly, since $1-t<0$ and $1\le r<\ell-1$, we get
\begin{align*}
	f(\ell,r,t,r) & = \left(t\binom \ell 2 + (t-1)\binom{r+1}2 \right)\left(\ell t - \ell + r-1\right) \\
	& - \left((t-1){\binom\ell 2} + (t-2){\binom{r+1}2} + r^2\right)\left(\ell t-1 \right) \\
	& = \binom\ell 2 \left(rt - 1 \right) + \binom{r+1}2\left(rt + \ell - r -1\right) - \ell r^2 t + r^2 \\
	& = \frac12\left(\ell^2 rt - \ell^2 - \ell rt + \ell + r^3t  + \ell r^2  -r^3 -r^2 + r^2t +\ell r -r^2 -r -2\ell r^2t + 2r^2\right) \\
	& = \frac12\left((\ell r^2 - r^3) + (-\ell r^2 t+r^3t) + (\ell^2 rt-\ell r^2 t) + (-\ell rt + r^2t) + (-\ell^2 + \ell r) + (\ell-r) \right) \\
	& = \frac{\ell-r}{2}\left(r^2-r^2 t + \ell rt - rt - \ell + 1  \right) \\
	& = \frac{\ell-r}{2}\left(r^2(1-t) + (\ell-1)rt - \ell +1 \right) \\
	& > \frac{\ell-r}{2}\left((\ell-1)r(1-t) + (\ell-1)rt - \ell +1 \right) \\
	& = \frac{\ell-r}{2}(\ell-1)(r-1) \geq 0.
\end{align*}

{\bf Case $r+1 \leq x \leq \ell-1$: } This time we would like to show that
\[d_t = \frac{t\binom \ell 2 + (t-1)\binom{r+1}2}{t\ell-1} > \frac{(t-1){\binom \ell 2} + (t-1){\binom {r+1} 2} + \binom x 2}{(t-1)\ell + x-1} = d_H,\]
where the second equality uses the fact that by our assumptions $H$ consists of $t-1$ copies of $K_{\ell}$ and one copy of $K_x$, and since $x\geq r+1$ each two consecutive cliques in $H$ are joined by an $r$ bridge.

Consider the function
\begin{align*}
	g(\ell,r,t,x) = \left(t\binom \ell 2 + (t-1)\binom{r+1}2 \right)\left((t-1)\ell + x-1\right) \\
	- \left((t-1){\binom \ell 2} + (t-1){\binom {r+1} 2} + \binom x 2\right)\left(t\ell -1 \right).
\end{align*}

Once again we will show that $g(\ell,r,t,x)>0$ for $\ell,r,t,x$ satisfying our assumptions. Before we proceed, notice that
\begin{align*}
	g(\ell,r,t,x) & = \binom \ell 2 (tx-1) + \binom {r+1}2(t-1)(x-\ell) - \binom x 2 (\ell t -1)\\
	& = \frac12 \left( (\ell^2 t x - \ell^2 - \ell t x + \ell) \right. \\
	& \ \ \ \ \ + (r^2tx - \ell r^2t - r^2 x + \ell r^2 + rtx - \ell rt - r x + \ell r ) \\
	& \ \ \ \ \ + \left. (-\ell tx^2 + x^2 +\ell t x - x  ) \right) \\
	& = \frac12 \left((\ell^2tx - \ell tx^2) + (-\ell r^2t + r^2tx) + (\ell r^2 - r^2x) + (-\ell rt + rtx) \right. \\
	& \ \ \ \ \ + \left. (-\ell^2 + \ell x) + (\ell r - rx) + (-\ell x + x^2) + (\ell - x) \right) \\
	& = \frac{\ell-x}2\left(\ell tx - r^2t + r^2 - rt - \ell + r - x + 1 \right).
\end{align*}
Hence, it is enough to show that
\begin{align*}
	h(\ell,r,t,x) = \ell tx - r^2t + r^2 - rt - \ell + r - x + 1 > 0.
\end{align*}
As before, the function $h$ is linear with respect to $x$.
Furthermore, since $\ell t -1>0$ this function is increasing with respect to~$x$ and so $h(\ell,r,t,x)\ge h(\ell,r,t,r+1)$. Thus, one only needs to check whether
$h(\ell,r,t,r+1)>0$. Since $r\leq\ell-2$ and $1-t<0$, we have
\begin{align*}
	h(\ell,r,t,r+1) & = \ell t(r+1) - r^2t + r^2 - rt - \ell \\
	& = (\ell-1)rt + r^2(1-t) + \ell(t-1) \\
	& \geq (\ell-1)rt + (\ell-2)r(1-t) + \ell(t-1) \\
	& = (r+t-1)\ell + r(t-2) > 0.
\end{align*}
This completes the proof of Proposition~\ref{<d}. \qed

\end{document}